\documentclass[review]{elsarticle}
\usepackage{framed}
\addcontentsline{toc}{section}{Acknowledgement}
\usepackage{geometry}
\usepackage{lineno,hyperref}
\usepackage{amssymb,amsmath,amsthm}
\usepackage{mathrsfs}
\numberwithin{equation}{section}
\journal{Journal of \LaTeX\ Templates}
\usepackage{graphicx}
\usepackage{subfigure}
\usepackage{float}
\usepackage{color}
\usepackage{epstopdf}

\newtheorem{theorem}{Theorem}[section]
\newtheorem{lemma}{Lemma}[section]

\newtheorem{property}[theorem]{Property}
\newtheorem{rem}{Remark}[section]
\newtheorem{thm}{Theorem}[section]

\newtheorem{lem}{Lemma}[section]
\newtheorem{prop}{Proposition}[section]

\newtheorem{definition}{Definition}[section]

\def\<{\left\langle} \def\>{\right\rangle}
\def\\left({\left\left(} \def\\right){\right\right)}

\begin{document}
	
	\begin{frontmatter}

		\title{	{Dynamics and spreading speeds of a   
		nonlocal diffusion model with advection and free boundaries }
			 \tnoteref{mytitlenote}}

		\author{Chengcheng Cheng}
		\ead{chengchengcheng@amss.ac.cn}

		\address{School of Mathematical Sciences, 
			Laboratory 
			of 
			Mathematics 
			and 
			Complex Systems, Ministry of Education, Beijing Normal University, Beijing 
			100875, China}

	
\begin{abstract}
In this paper, we investigate a Fisher-KPP nonlocal diffusion model incorporating the effect of advection and free boundaries, aiming to explore the propagation dynamics of the nonlocal diffusion-advection model. Considering the effects of the advection, the existence, uniqueness, and regularity of the global solution are obtained.  We introduce the principal eigenvalue of the nonlocal operator with the advection term and discuss the asymptotic properties influencing the long-time behaviors of the solution for this model. Moreover, we give several sufficient conditions determining the occurrences of spreading or vanishing and obtain the spreading-vanishing dichotomy. Most of all, applying the semi-wave solution and constructing the upper and the lower solution,  we give an explicit description of the finite asymptotic spreading speeds for the double free boundaries on the effects of the nonlocal diffusion and advection compared with the corresponding problem without an advection term.
\end{abstract}
\end{frontmatter}

\section{Introduction}\label{s1}
In the past years, the nonlocal diffusion equation has been recognized to better describe the long-distance dispersal of species and propagation of epidemics,  such as \cite{furter1989local, andreu2010nonlocal, kao2010random, massaccesi2017nonlocal} and so on. The researches on nonlocal problems have attracted widespread attention.  Berestycki et al.~\cite{berestycki2009non}  considered the Fisher-KPP equation with a nonlocal saturation effect and mainly studied the existence of the steady
state and the traveling waves. Later, the persistence criteria for populations with nonlocal
diffusion were analyzed in \cite{berestycki2016persistence}. In the natural environment, the migrations of the species and epidemic spreading usually change with time. 
The free boundary used to describe the migration and spreading frontiers is more reasonable for dynamical studies in reality.   The investigations of the nonlocal diffusion model with free boundaries met great developments (see \cite{zhao2020dynamics, du2020analysis, li2020systems, pu2021west, li2022free, du2022two, li2022dynamics, du2022high} and the references therein).

The advection movements (e.g., the wind direction and the water flow) play a significant role in epidemic dispersal \cite{gu2014long, ge2015sis, cui2016spatial,  tian2018, cheng2021dynamics} and species survival in the river environment \cite{lutscher2005effect, lutscher2006effects, JIANG2021103350, yan2022competition}. Especially, Maidana and Yang found out that the West Nile virus spread from New York to California, which was recognized as a consequence of the diffusion and advection movements of birds~\cite{maidana2009spatial}. However, few studies have formally explored the important effects of the advection on nonlocal diffusion problems over the past decades.
It is increasingly reasonable to introduce the advection to study the propagation dynamics of the nonlocal diffusion model. As is well known,
the spreading speed is one of the most concernings in the research for epidemic propagation. 
The classical work on studying the spreading speed of the reaction-diffusion equation  is that Fisher~\cite{fisher1937wave} and Kolmogoroff, Petrovsky,
and Piscounoff \cite{kolmogorov1937etude} in 1937 considered the following reaction-diffusion equation
\begin{equation}\label{ff}
\begin{cases}
u_t=u_{xx}+f(u),&t>0, ~x\in\mathbb R,\\
u(0,x)=u_0(x),&x\in \mathbb R.
\end{cases}
\end{equation}
They proved the existence of
 minimal wave speed $c_0^*>0$ such that $(\ref{ff})$ possessed a  traveling
wave solution  $$u(x,t)=\omega (x-ct) \text{ if and only if } c>c_0^*.$$
Weinberger, Lewis, and their collaborators investigated the traveling waves and spreading speeds of cooperative models \cite{li2005spreading, weinberger2007anomalous}  and competition models \cite{lewis2002spreading} in a series of meaningful works.  Recently, Liang and Zhou~\cite{liang2022propagation} discussed the spreading speed of the positive and negative directions for local KPP equations with advection in an almost periodic environment.  For the nonlocal Fisher-KPP model with free boundaries,  Du et al. in \cite{du2021semi} first investigated the boundary spreading speeds by applying the semi-wave and traveling wave and they found the threshold condition on the kernel function which determines when the spreading speed is finite.  A more interesting question that attracts us is the difference in the asymptotic spreading speed between the leftward front and the rightward front when the spreading occurs.

Many of these previous studies have made great progress, whereas the dynamics and spreading speeds for the nonlocal diffusion problem with the advection and free boundaries have not been systematically explored yet.  In this paper, we first investigate the propagation dynamics and spreading speeds of the following nonlocal diffusion problem with the advection term and free boundaries.
\begin{equation}\label{eq1-1}
	\left\{\begin{array}{ll}
		u_{ t}=d \int_{g\left(t\right)}^{h\left(t\right)} 
		J\left(x-y\right) u\left(t, y\right) \mathrm{d} y-d~ 
		u-\nu u_{x}+f(t,x,u), & t>0,~
		g\left(t\right)<x<h\left(t\right), \\
		h^{\prime}\left(t\right)= \mu
		\int_{g\left(t\right)}^{h\left(t\right)} 
		\int_{h\left(t\right)}^{\infty} J\left(x-y\right) u\left(t, 
		x\right) \mathrm{d} y\mathrm{d} x, 
		& t 
		\geq 0, \\
		g^{\prime}\left(t\right)=-\mu
		\int_{g\left(t\right)}^{h\left(t\right)} 
		\int_{-\infty}^{g\left(t\right)} J\left(x-y\right) u\left(t, 
		x\right) \mathrm{d} y\mathrm{d} 
		x, & 
		t \geq 0,\\
		u\left(t, g\left(t\right)\right)=u\left(t, 
		h\left(t\right)\right)=0, & t \geq 0, \\
		 u\left(0, x\right)=u_{ 0}\left(x\right),
		 & x\in[-h_{0},h_0],\\
   h\left(0\right)=h_{0},~g\left(0\right)=-h_{0}.
	\end{array}\right.
\end{equation}
The model $(\ref{eq1-1})$ can be seen as a nonlocal extension for the diffusion-advection model with free boundaries in~\cite{gu2014long}.    
Here, $u(t, x)$ represents the population density; $d$ is the diffusion rate; $\nu$ represents the advection rate; $\mu$ denotes the boundary expanding ability. 
$g(t)$ and $h(t)$ denote the leftward spreading front and rightward spreading front, respectively. 
 $\int_{g\left(t\right)}^{h\left(t\right)} 
		\int_{-\infty}^{g\left(t\right)} J\left(x-y\right) u\left(t, 
		x\right) \mathrm{d} y\mathrm{d} 
		x$ and $\int_{g\left(t\right)}^{h\left(t\right)} 
		\int_{h\left(t\right)}^{\infty} J\left(x-y\right) u\left(t, 
		x\right) \mathrm{d} y\mathrm{d} x$  are the outward flux of the double fronts.
We always assume that $u(t, x)=0, \text{ for } x\notin [g(t), h(t)],$ and the boundary expanding ratio is proportional to the outward flux.

Further, assume that the initial function $u_0(x)$ satisfies 
\begin{equation}\label{eq1-03}
	u_{0}(x) \in C^1\left(\left[-h_{0}, h_{0}\right]\right),~ u_{0}\left(-h_{0}\right)=u_{0}\left(h_{0}\right)=0,~ u_{0}(x)>0,~ x\in\left(-h_{0}, h_{0}\right).
\end{equation}

The kernel function $J(x): \mathbb R\rightarrow \mathbb R$ satisfies the following conditions.
\newline $(\mathbf{J})$  $J \in C^1(\mathbb{R}) \cap L^{\infty}(\mathbb{R}), $    $J \geq 
0,$   $J(x)=J(-x),$  $ 
J(0)>0,$ and $\int_{\mathbb{R}} J(x) d x=1.$

Assumptions on the reaction term $f: \mathbb{R}^{+} \times\mathbb{R}\times\mathbb{R}^{+}\rightarrow \mathbb{R}.$
\newline $(\textbf{f1})$ $f(t,x,0)=0$ and $f(t,x,u)$ is locally Lipschitz continuous in $u \in 
\mathbb{R}^{+}$. That is, for any $U^*>0$, there is a constant $k=k(U^*)>0$ such 
that
$
\left|f\left(t,x,u_{1}\right)-f\left(t,x,u_{2}\right)\right| \leq 
k\left|u_{1}-u_{2}\right| \text { for } u_{1}, u_{2} \in[0, U^*].
$
\newline $(\textbf{f2})$ There is $k_{0}>0$ such that $f(t,x,u)<0$ for $u \geq 
k_{0}$.
\newline$(\textbf{f3})$ For any given $\tau,$ $l,$ $u^*>0$, there is a constant $\tilde {k}\left(\tau, l, u^*\right)$ and $\alpha\in(0,1)$ such that
$$
\left\|f_{}(\cdot, x, u)\right\|_{C^{\alpha}([0, \tau])} \leq \tilde {k}\left(\tau, l, u^*\right)
$$
~~~~~~for  $x \in[-l, l], ~u \in\left[0, u^*\right]$.

 Assume that $f$ is of Fisher-KPP type as exploring the long-time asymptotic behaviors,  that is, the following additional conditions for  $f$ hold.
\newline$(\textbf{f4})$ 
$f =f(u) \in  C^1$ independent of $(t,x)$.  $f(0)=f(1)=0<f(x), \text{ for } x\in (0,1).$ 
\newline $\qquad~~~~~~~ f^{\prime}(0):=f_0>0>f^{\prime}(1)$,  and $\dfrac{f(u)}{u}$ is nonincreasing in 
$u>0$.

Additionally, considering the small advection rate and the spreading propagation criterion in Section \ref{s5}, we always assume that 
\begin{equation}\label{eq1-2-2-2}
0<\nu<\tilde c,
\end{equation}
where 
$\tilde c$ is exactly the minimal spreading speed 
of 
the traveling wave solution of the following equation
\begin{equation}\label{eq1-2004}
	\begin{cases}
		u_{t}=d \int_{\mathbb{R}} J\left(x-y\right) u(t, y){\rm d} y-d~ u(t, x)
		+f(t,x, u), & t>0,~x 
		\in \mathbb{R}, 
		\\ u(0, x)=u_{0}(x), & x \in \mathbb{R}.
	\end{cases}
\end{equation}

 In this paper, we mainly investigate the propagation dynamics and the spreading speed of a nonlocal diffusion model $(\ref{eq1-1})$ with the advection and free boundaries which are firstly systematically discussed. 
The existence, uniqueness, and regularity, the long-time dynamical behavior, and the asymptotic spreading speeds for $(\ref{eq1-1})$ are investigated.  In general, the solution $u(t, x )$ for a nonlocal model may not be differentiable in $x$. In our model, considering the effect of the advection described by a gradient term, we can obtain the regularity of the solution $u(t, x)$ in $x$. Further, we define the corresponding principal eigenvalue with the advection term and show that the advection can influence the spreading or vanishing by affecting the value of the principal eigenvalue. Especially, the spreading speed plays an important role in predicting the epidemic propagation speed and scale. Under the effects of advection and nonlocal diffusion, we construct sharp estimates of the spreading speeds of the double free boundaries. We show that the leftward front and the rightward front move at different spreading speeds compared with the non-advection case. 

The remainder of the paper is organized as follows.  In Section \ref{s2},  based on the contraction mapping theorem and the associated Maximum principle, we prove that $(\ref{eq1-1})$ admits a
unique global solution defined for all $t > 0$.  In Section \ref{s3}, we define the principal eigenvalue associated with the corresponding nonlocal operators and then obtain the asymptotic properties.  In Section \ref{s4}, under the conditions $(\mathbf{J})$ and $(\mathbf{f1})-(\mathbf{f4})$, we study the long-time dynamical behavior of the solution to $(\ref{eq1-1})$, and obtain the spreading-vanishing dichotomy regimes. In Section \ref{s5}, we investigate the effects of the advection on the spreading speeds and get sharp estimates by constructing the upper solution and the lower solution and applying the corresponding semi-wave solution.
\section{Existence and uniqueness}\label{s2}
For given $h_0>0,$ and any $0<T<\infty,$  denote 
\begin{equation} 
   \begin{aligned}
&\mathscr D_{h_0}:=[0, T] \times\left[-h_{0}, h_{0}\right],  
\\&\mathbb{H}_{h_0, T}:=\left\{h \in C([0, T])\mid h(0)=h_0, \inf _{0 \leq t_1<t_2 \leq T} \frac{h\left(t_2\right)-h\left(t_1\right)}{t_2-t_1}>0\right\},
\\&\mathbb{G}_{h_0, T}:=\left\{g \in C([0, T])\mid -g \in \mathbb{H}_{h_0, T}\right\},
\\&\mathcal D_{T}:=D_{g, h}^{T}=\left\{(t, x) \in \mathbb{R}^{2}\mid 0<t \leq T, ~g(t)<x<h(t)\right\}.
\end{aligned}
\end{equation}
Now we present the existence and uniqueness of the solution for $(\ref{eq1-1})$ in this section.
\begin{theorem}\label{th2-1}
Assume that $(\mathbf{J})$  and $(\mathbf{f1})-(\mathbf{f3})$ hold, for any initial function satisfies $(\ref{eq1-03})$ and for  $T\in(0,\infty)$, the problem $(\ref{eq1-1})$ admits a unique  solution 
$$(u,g,h)\in C^{1+\alpha, 1}(\overline D^T_{g,h})\times C^{1+\alpha}([0,T])\times C^{1+\alpha}([0,T]).$$
\end{theorem}

We first show the following  Maximum principle, which plays a vital part in proving the positivity of $u(t, x).$
\begin{lemma} [Maximum principle]\label{lemm2.1} 
 Assume that ${(\mathbf J)}$ holds, and $g \in \mathbb{G}_{h_{0}, T}, $ $h \in \mathbb{H}_{h_{0}, T}$ for some $h_{0},$ $T>0$. Suppose that  
 $ u_x$, $ u_t \in C (\overline D_{{g}, {h}}^T )$  and  $c(t,x) \in L^{\infty}(D_{g, h}^{T})$, satisfy that
\begin{equation}\label{eq3-8}
\begin{cases}u_{t}(t, x) \geq d \int_{g(t)}^{h(t)} J\left(x-y\right) u(t, y) {\rm d} y-d ~u-\nu u_x+c(t, x) u, & t \in(0, T], ~x \in(g(t), h(t)), \\ u(t, g(t)) \geq 0, ~u(t, h(t)) \geq 0, & t>0, \\ u(0, x) \geq 0, & x \in[-h_0,h_0],
\end{cases}
\end{equation}
Then $u(t, x) \geq 0$ for all $0 \leq t \leq T$ and $x \in[g(t), h(t)]$. Moreover, if  $u(0, x) \not \equiv 0$ in $\left[-h_{0}, h_{0}\right]$, it follows that $u(t, x)>0$ in $\mathcal D_{T}$.
\end{lemma}
\begin{proof}
Let $w(t, x)=e^{-\xi t} u(t, x)$ with  $\xi>\|c(t,x)\|_{L^{\infty}(\mathcal D_T)}.$
Then
\begin{equation}\label{eq3-10}
\begin{aligned}
&w_{t}(t,x) +\nu w_x(t,x) -d \int_{g(t)}^{h(t)} J\left(x-y\right) w(t, y) {\rm d } y+d w(t,x)
\\& \geq -\xi w(t,x)+c(t,x) w(t,x),  \text{ for } (t, x)\in \mathcal D_{T}.
\end{aligned}
\end{equation}

Now we wish to prove that $w \geq 0$ in $\mathcal D_{T}$. On the contrary, for $0<T^*\leq T$, assume that
$
\min\limits_{(t, x) \in   \overline {\mathcal  D}_{T^* } }w(t, x)<0.
$
According to $(\ref{eq3-8})$, $w \geq 0$ on the parabolic boundary of $\mathcal D_{T}$, and  then there exists $\left(t^{*}, x^{*}\right) \in \mathcal D_{T^*}$ such that 
$
w\left(t^{*}, x^{*}\right)=\min\limits_{(t, x) \in   \overline {\mathcal  D}_{T^* } }w(t, x).
$
 
 Thus, 
 \begin{equation}
\begin{aligned}
&0\geq w_{t}(t^*,x^*) +\nu w_x(t^*,x^*) -d \int_{g(t^*)}^{h(t^*)} J(x^*-y) w(t^*, y) {\rm d} y+d w(t^*,x^*)
\\& \geq -\xi w\left(t^{*}, x^{*}\right)+c(t^*,x^*) w\left(t^{*}, x^{*}\right),
\\&>0 , \text{ for } (t, x)\in \mathcal D_{T^*},
\end{aligned}
\end{equation}
 which yields a contradiction.
 It follows that $w(t, x) \geq 0$ and thus $u(t, x) \geq 0$ for all $(t, x) \in \mathcal D_{T^{*}}$.

If $T^{*}=T$, then $u(t, x) \geq 0$ in $\mathcal D_{T}$; while if $T^{*}<T$, we can finitely repeat this process with $u_{0}(x)$ replaced by $u\left(T_{*}, x\right)$ and $(0, T]$ replaced by $\left(T^{*}, T\right]$ to complete the proof.

If $u(0, x) \not \equiv 0, \text{ in } \left[-h_{0}, h_{0}\right],
$
 the remaining  only needs to show that $w>0$ in $\mathcal D_{T}$.
Motivated by the arguments in proving  Lemma~2.2~\cite{cao2019}, it is not difficult to complete the proof by contradiction.
\end{proof}

The following result can be proved by the arguments in the proof of {Maximum principle}. 
\begin{lemma}[Comparison criterion] \label{lemm2.2}
	 Assume that $(\mathbf{J})$ holds. For any $h_{0}>0,$  $T>0$, suppose that  $u_x(t,x)$,  $u_{t}(t, x)\in C(\mathscr D_{h_0})$, and  $c(t,x) \in L^{\infty}\left(\mathscr D_{h_0}\right)$, satisfy that
$$
\begin{cases}u_{t}(t, x) \geq d \int_{-h_{0}}^{h_{0}} J\left(x-y\right) u(t, y) {\rm d} y-d~ u-\nu  u_x+c(t, x) u, & t \in(0, T],~x \in\left[-h_{0}, h_{0}\right], \\ u(0, x) \geq 0, & x \in\left[-h_{0}, h_{0}\right] .
\end{cases}
$$
Then $u(t, x) \geq 0$ for all $(t, x) \in \mathscr D_{h_0}$. Moreover, if $u(0, x) \not \equiv 0$ in $\left[-h_{0}, h_{0}\right]$, then $u(t, x)>0$ in $(0, T] \times\left[-h_{0}, h_{0}\right]$.
\end{lemma}
According to Lemmas~\ref{lemm2.1} and \ref{lemm2.2}, we have the following result.
\begin{lemma}[Comparison principle]\label{th3-3.3}
  Assume that $(\mathbf{J})$ holds. Let $T>0,$ $\bar{h},$ $ \bar{g} \in C^1([0, T]).$ Suppose that  $ \bar u_x,$  $ \bar u_t \in C(D_{\bar{g}, \bar{h}}^T)$  and satisfy
\begin{equation}\label{eq2-12}
\begin{cases}\bar{u}_t \geq d \int_{\bar{g}(t)}^{\bar{h}(t)} J\left(x-y\right) \bar{u}(t, y) \mathrm{d} y-d~\bar{u}-\nu \bar u_x+f(\bar{u}), & (t, x) \in D_{\bar{g}, \bar{h}}^T, 
 \\ \bar{h}^{\prime}(t) \geq \mu \int_{\bar{g}(t)}^{\bar{h}(t)} \int_{\bar{h}(t)}^{\infty} J\left(x-y\right) \bar{u}(t, x) \mathrm{d} y \mathrm{d} x,  & 0<t \leq T, \\ \bar{g}^{\prime}(t) \leq \mu \int_{\bar{g}(t)}^{\bar{h}(t)} \int_{-\infty}^{\bar{g}(t)} J\left(x-y\right) \bar{u}(t, x) \mathrm{d} y \mathrm{d} x, & 0<t \leq T, 
 \\ \bar{u}(t, \bar{g}(t)) \geq 0, ~\bar{u}(t, \bar{h}(t)) \geq 0, & 0<t \leq T,
 \\ \bar{u}(0, x) \geq u_0(x) , & x\in[- h_0, h_0], 
 \\  \bar{g}(0) \leq-h_0,~\bar{h}(0) \geq h_0. & \end{cases}
\end{equation}
Let $(u, g, h)$ be the unique solution of $(\ref{eq1-1})$,  we have
$$
u \leq \bar{u}, ~g \geq \bar{g}, ~h\leq \bar{h} ~\text { in } D_{g, h}^T.
$$
\end{lemma}
\begin{rem}
 $(\bar u, \bar g, \bar h)$ is called as the upper solution of $(\ref{eq1-1})$. Clearly, the lower solution $(\underline u, \underline g, \underline h)$ of $(\ref{eq1-1})$ can be similarly defined by reversing the inequalities of $(\ref{eq2-12})$.
\end{rem}
Now we complete the proof of  Theorem~$\ref{th2-1}.$
\begin{proof}[The proof of Theorem~$\ref{th2-1}$]
Let 
\begin{equation}\label{eq2-03}
    y=\dfrac{2x}{h(t)-g(t)}-\dfrac{h(t)+g(t)}{h(t)-g(t)},~w(t,y)=u(t,x),
\end{equation}
which transforms $x\in(g(t),h(t))$ into $y\in(-1,1),$ then
\begin{equation}\label{eq2-3}
   \begin{aligned}
&\dfrac{\partial y}{\partial x}=\dfrac {2}{h(t)-g(t)}\\&\quad:=\mathscr A(t,g(t),h(t)), \\
&\dfrac{\partial y}{\partial t}= -\frac{y\left(h^{\prime}(t)-g^{\prime}(t)\right)+\left(h^{\prime}(t)+g^{\prime}(t)\right)}{h(t)-g(t)} \\
&\quad:=\mathscr B\left(y, g(t), g^{\prime}(t), h(t), h^{\prime}(t)\right).
\end{aligned}
\end{equation}
Denote $$
\begin{aligned}
& h_*:=\mu\int_{-1}^{1}\int_{1}^{\infty}J(h_0(y-z))u_0(h_0 y)h_0^2{\rm d}z {\rm d}y,\\
& g_*:=-\mu\int_{-1}^{1}\int_{-\infty}^{-1}J(h_0(y-z))u_0(h_0 y)h_0^2{\rm d}z {\rm d}y.
\end{aligned}$$ 
For $0<T<\min\left\{1, \dfrac{h_0}{4(1+h_*)}, \dfrac{h_0}{4(1-g_*)}\right\} $, set $\Delta_T:=[0,T]\times[-1,1],$
$$
\begin{aligned}
&\mathcal{D}_{h, T}:=\left\{h \in C^{1}([0, T])\mid h(0)=h_{0},~h^{\prime}(0)=h_{*},~\left\|h^{\prime}-h_{*}\right\|_{C([0, T])} \leq 1\right\},\\
&\mathcal{D}_{g, T}:=\left\{g \in C^{1}([0, T])\mid g(0)=-h_{0},~g^{\prime}(0)=g_{*},~\left\|g^{\prime}-g_{*}\right\|_{C([0, T])} \leq 1\right\}.
\end{aligned}
$$

Given $h_{i} \in \mathcal{D}_{h,T}$,~$g_{i} \in \mathcal{D}_{g,T},~i=1,2$, for $h_{i}(0)=h_{0}$,  $g_i(0)=-h_0,$ we can obtain that
\begin{equation}\label{eq2-08}
\begin{aligned}
&\left\| h_1-h_{2}\right\|_{C([0, T])} \leq T\left\|h_1^{\prime}-h_{2}^{\prime}\right\|_{C([0, T])},\\
&\left\|g_1-g_2\right\|_{C([0, T])} \leq T\left\|g_{1}^{\prime}-g_{2}^{\prime}\right\|_{C([0, T])}. 
\end{aligned}
\end{equation}
For $h\in\mathcal{D}_{h,T},~g\in\mathcal{D}_{g,T},$ it follows 
$$\|h-h_0\|\leq T(1+h_*)\leq \dfrac{h_0}{4},~\text{ and } 
\|g+h_0\|\leq T(1-g_*)\leq \dfrac{h_0}{4},$$
then the translations $(\ref{eq2-03})-(\ref{eq2-3})$ are well defined.

Further, one can see that $w(t,y)$ satisfies 
\begin{equation}\label{eq2-6}
\begin{cases}
w_t=d\int^{1}_{-1}J(\frac{y-z}{\mathscr A})w(t,z)\frac{1}{\mathscr A}{\rm d}z-d~ w(t,y)
+\left(\nu \mathscr A+\mathscr B\right)w_y+f(t, y, w), &t>0,~ y \in(-1,1), \\ 
w( t,1)=w(t,-1)=0,  & t>0, \\ 
w(0,y)=u_{0}\left(h_{0} y\right),  & y \in[-1,1].
\end{cases}
\end{equation}
{According to   the standard $L^p$ theory and the Sobolev embedding theorem \cite{bates2007} with the arguments of Section~4 in \cite{andreu2010nonlocal}, in view of $(\mathbf{J})$ and $(\mathbf{f1})-(\mathbf{f3})$, the equation $(\ref{eq2-6})$ admits a unique solution $$w(t,y)\in C^{1+\alpha, 1}\left(\Delta_T\right)  \text{ and } \|w(t,y)\|_{C^{1+\alpha, 1}\left(\Delta_T\right)}\leq C,$$
where positive constant $C$ depends on $h_0, \alpha$ and $ \|u_0\|_{C^{1}\left([-h_0,h_0]\right)}.$}

Moreover, for the above $w$, define 
\begin{equation}\label{eq2-7}
\begin{aligned}
    &h(t)=h_0+\mu\int^{t}_{0}\int^{1}_{-1}\int^{\infty}_{1} J(\frac{y-z}{\mathscr A_0})w(\tau,y)\frac{1}{\mathscr A_0^2}{\rm d}z {\rm d}y {\rm d}\tau,\\&
    g(t)=-h_0-\mu\int^{t}_{0}\int^{1}_{-1}\int^{-1}_{-\infty} J(\frac{y-z}{\mathscr A_0})w(\tau,y)\frac{1}{\mathscr A_0^2}{\rm d}z {\rm d}y {\rm d}\tau,
    \end{aligned}
\end{equation}
where $\mathscr A_0=\dfrac {2}{h(\tau)-g(\tau)}, ~\tau\in[0,t].$
  
  Since $w$ depends on $h$ and $g$, the problem (\ref{eq2-7}) has a unique solution $(\tilde h, \tilde g)$, where $\tilde{h}(t)=\tilde{h}(t; g,h),~ \tilde{g}(t)=\tilde{g}(t; g,h)$. Then $\tilde{h}(0)=h_{0},~\tilde g(0)=-h_0,$ 
  \begin{equation}\label{eq2-8}
 \begin{aligned}
  &\tilde{h}^{\prime}(t)=\mu\int^{1}_{-1}\int^{\infty}_{1} J(\frac{y-z}{\mathscr A})w(t,y)\frac{1}{\mathscr A^2}{\rm d}z {\rm d}y,
  \\&\tilde h^{\prime}(0)=\mu\int^{1}_{-1}\int^{\infty}_{1} J(\frac{y-z}{\mathscr A})w(0,y)\frac{1}{\mathscr A^2}{\rm d}z {\rm d}y,\\
  &\tilde{g}^{\prime}(t)=-\mu\int^{1}_{-1}\int^{-1}_{\infty} J(\frac{y-z}{\mathscr A})w(t,y)\frac{1}{\mathscr A^2}{\rm d}z {\rm d}y,
  \\&\tilde g^{\prime}(0)=-\mu\int^{1}_{-1}\int^{-1}_{-\infty} J(\frac{y-z}{\mathscr A})w(0,y)\frac{1}{\mathscr A^2}{\rm d}z {\rm d}y.
  \end{aligned}
  \end{equation}
  And for any  $ h\in\mathcal{D}_{h,T},$ $g\in\mathcal{D}_{g,T}$, one sees that 
 $\tilde{h}\in C^{1+\alpha }([0, T]),$  continuously depends on $w \in C^{1+\alpha, 1}\left(\Delta_{T}\right)$, $h\in \mathcal{D}_{h,T},$ and $g\in \mathcal{D}_{g,T}$. So does $\tilde g$. Then there is $\hat C>0$ such that
 \begin{equation}\label{eq2-9}
\begin{aligned}
 &\tilde{h}^{\prime} \in C^{\alpha}([0, T])\text{ and } \left\|\tilde{h}^{\prime}\right\|_{C^{\alpha}([0, T])} \leq \hat C,\\
 &\tilde{g}^{\prime} \in C^{\alpha}([0, T])\text{ and } \left\|\tilde{g}^{\prime}\right\|_{C^{\alpha}([0, T])} \leq \hat C. 
 \end{aligned}
\end{equation}
  
  Define $\mathcal{F}: \mathcal{D}_{h,T}\times\mathcal{D}_{g,T} \longrightarrow C^{1}\left([0, T]\right)\times C^{1}([0, T])$ by
\[
\mathcal{F}(h,g)=(\tilde{h},\tilde {g}).
\]
It can be seen that $\mathcal{F}$ is continuous in $\mathcal{D}_{h,T}\times\mathcal{D}_{g,T}$, and $(g,h) $ is a fixed point of $\mathcal{F}$ if and only if $(w; h, g)$ solves (\ref{eq2-6}) with (\ref{eq2-7}). According to (\ref{eq2-8})-(\ref{eq2-9}),  it follows that $\mathcal{F}$ is compact and satisfies 
$$
\begin{aligned}
\|\tilde{h}^{\prime}-h^{*}\|_{C([0, T])}+
\|\tilde{g}^{\prime}-g^{*}\|_{C([0, T])} \leq \left ( \| h^{\prime} \|_{C^{\alpha}([0,T])} +\| g^{\prime} \|_{C^{\alpha}([0,T]) } \right)
 T^{\alpha } \leq \mathscr C T^{\alpha },
\end{aligned}
$$
where positive constant $\mathscr C\geq 2\hat C$.
Therefore, if
\begin{equation}\label{eq-T}
T < \min \left\{1,~\dfrac{h_0}{4(1+h_*)},~ \dfrac{h_0}{4(1-g_*)},~ \mathscr C ^{-\frac{1}{\alpha}}\right\},
\end{equation}
we can obtain that $\mathcal{F}$ maps $\mathcal{D}_{h,T}\times\mathcal{D}_{g,T}$ into itself.

Next, we need to show that $\mathcal {F}$ is a contraction mapping on $\mathcal{D}_{h,T}\times \mathcal{D}_{g,T}$.

For any $w_1,$ $w_2$ defined on $ \Delta _T$ and  satisfying $(\ref{eq2-6})$, set $\omega:=w_1-w_2,$  it follows 
$$
\begin{cases}
\omega_t=d\int_{-1}^{1}J(\frac{y-z}{\mathscr A})\omega(t, z)\frac{1}{\mathscr A}{\rm d}z-d~ \omega(t, y)+\left(\nu \mathscr A+\mathscr B\right)\omega_y+f(t, y, w_1)-f(t, y, w_2),&t>0, ~y\in(-1,1),\\
\omega(t,1)=\omega(t,-1)=0,&t>0,\\
\omega(0,y)=0,&y\in[-1,1].
\end{cases}
$$
Using the $L^{p}$ estimates  for partial differential equations and Sobolev embedding theorem, we obtain
\begin{equation}\label{eq2-023}
\left\|{w}_{1}-{w}_{2}\right\|_{C^{1+\alpha, 1}\left(\Delta_{T}\right)} \leq {\tilde{C}}\left(\left\|w_{1}-w_{2}\right\|_{C\left(\Delta_{T}\right)}+\left\| h_{1}-h_{2}\right\|_{C^{1}([0, T])}+\left\|g_{1}-g_{2}\right\|_{C^{1}([0, T])}\right),
\end{equation}
where ${\tilde{ C}}$ depends on $\hat C$ and the functions $\mathscr A$ and $\mathscr B$ in $(\ref{eq2-3})$. According to $(\ref{eq2-8})$ and $(\ref{eq2-9})$, it gives
\begin{equation}
\begin{aligned}\label{eq2-22}
&\left\|\tilde{h}_{1}^{\prime}-\tilde{h}_{2}^{\prime}\right\|_{C^{\alpha}([0, T])} \leq \mu \mathscr C_* \left\|{w}_{1}-{w}_{2}\right\|_{C\left(\Delta_{T}\right)},\\
&\left\|\tilde{g}_{1}^{\prime}-\tilde{g}_{2}^{\prime}\right\|_{C^{\alpha}([0, T])} \leq \mu \mathscr C_* \left\|{w}_{1}-{w}_{2}\right\|_{C\left(\Delta_{T}\right)},
\end{aligned}
\end{equation}
where positive constant $\mathscr C_*$ depends on $T$ and $\alpha$.

Next, we give an estimate of $\|w_1-w_2\|_{C(\Delta_{T})}$. By $(\textbf{f1})$ and $(\ref{eq2-08})$ with $(\ref{eq2-023})$, direct calculus gives
\begin{equation}\label{eq2-23}
\begin{aligned}
 w_1-w_2&=d\int_{0}^{t}\int_{-1}^{1}J(\frac{y-z}{\mathscr A})(w_1-w_2)(\tau,z)\frac{1}{\mathscr A}{\rm d}z{\rm d}\tau-d\int_{0}^{t} (w_1-w_2)(\tau,y){\rm d}\tau\\&+\int_{0}^{t}(\nu \mathscr A+\mathscr B)(w_{1,y}-w_{2,y}){\rm d} \tau+\int_{0}^{t}f(\tau,y,w_1)-f(\tau,y,w_2){\rm d}\tau \\
 &\leq l (d, \mathscr A) T\|w_1-w_2\|_{C(\Delta_{T})}+(\nu \mathscr A+\mathscr B)T\left\|{w}_{1, y}-{w}_{2, y}\right\|_{C\left(\Delta_{T}\right)}+kT\|w_1-w_2\|_{C(\Delta_{T})}\\
 &\leq  \mathscr PT\left(\|h_1^{\prime}-h_2^{\prime}\|_{C([0,T])}+\|g_1^{\prime}-g_2^{\prime}\|_{C([0,T])}\right),
    \end{aligned}   
\end{equation}
{where} positive constants  $ l (d, \mathscr A) $ depends on $d, \mathscr A$ and $\mathscr P $  depends on $d,$ $\nu,$ $\mathscr A$, $\mathscr B$ and $T$.

Combining $(\ref{eq2-22})$ and $(\ref{eq2-23})$, we then can obtain that 
\begin{equation}
\begin{aligned}
    &\|\tilde h_1^{\prime}-\tilde h_2^{\prime}\|_{C([0,T])}+ \|\tilde g_1^{\prime}-\tilde g_2^{\prime}\|_{C([0,T])}
    \leq 2  \mu   \mathscr P \mathscr C_* T \left(\|h_1^{\prime}-h_2^{\prime}\|_{C([0,T])}+\|g_1^{\prime}-g_2^{\prime}\|_{C([0,T])}\right).
    \end{aligned}
\end{equation}
Choose $T$ which satisfies $(\ref{eq-T})$ such that $2  \mu  \mathscr P \mathscr C_*  T<1$,
we can get that $\mathcal {F}$ is a contraction mapping.
By the contraction mapping theorem, $\mathcal{F}$ admits a unique fixed point $(h,g) \in \mathcal{D}_{h,T}\times \mathcal{D}_{g,T}$, and problem $(\ref{eq2-6})$ with $(\ref{eq2-7})$ has a unique solution $(w; h, g)$. Moreover, by the Maximum principle,  
$$w(t, y)>0, \text{ for } t>0, y\in(-1, 1),$$  it follows 
\begin{equation}
h^{\prime}(t)>0 \text { and }g^{\prime}(t) <0, \text { for } t>0.
\end{equation}
Therefore, the function $u(t, x)=w\left(t,y\right)$ satisfies
$$
u \in  C^{1+\alpha, 1}(\overline D^T_{g,h}),~u>0 \text { in } D_{g,h}^T,
$$
and $(u;h,g)$ solves $(\ref{eq1-1})$ with
$
h,~g \in C^{1+\alpha}([0, T]).
$
\end{proof}
According to the proof of the Theorem~$\ref{th2-1},$ we immediately have the following result.
\begin{theorem}
For the assumptions given by Theorem~$\ref{th2-1}$,  let $(u,g,h)$ be the solution of $(\ref{eq1-1})$, then 
\begin{equation}\label{eq2-13}
h^{\prime}(t)>0 \text { and } g^{\prime}(t)<0, \text{ \rm for } t>0.
\end{equation}
\end{theorem}

Moreover, we intend to complete the global existence of this nonlocal diffusion equation with the advection.
\begin{theorem}
Under the assumptions of Theorem~$\ref{th2-1}$, the solution for $(\ref{eq1-1})$ exists for all $t>0$.
\end{theorem}
\begin{proof}
Now we prove that the unique solution of (\ref{eq2-6}) with (\ref{eq2-7}) defined over $0<t \leq T$ can be uniquely extended to all $t>0$. This extension can be done in a similar method as in Step 2 of the proof of Theorem~2.1 in \cite{wang2020free} arguing by contradiction.  Since only obvious modifications are needed,  the details are omitted.
\end{proof}
\section{Principal eigenvalue}\label{s3}
For any $h>0,$ $a_0>0,$ $\varphi\in C^1([-h, h]),$  we define the operator 
$\tilde {\mathcal L}:=\mathcal{L}_{(-h, h)}^{d}+a_0-\nu\nabla$ by
$$
\tilde {\mathcal L}[\varphi](x):=d\int_{-h}^{h } J\left(x-y\right) \varphi\left(y\right) \mathrm{d} 
y-d\varphi\left(x\right)-\nu\varphi^{\prime}\left(x\right)+a_0\varphi(x), ~ x \in[-h, 
h].
$$
The  principal eigenvalue of $\mathcal{L}_{(-h, h)}^{d}+a_0-\nu\nabla$ is given by 
\begin{equation}\label{lam}
\begin{aligned}
\lambda_p( \mathcal{L}_{(-h, h)}^{d}+a_0-\nu\nabla)= \inf \left\{ {\lambda \in \mathbb{R}\mid \left(\mathcal{L}_{(-h, h)}^{d}+a_0-\nu\nabla\right) {\Phi} \leq \lambda {\Phi} {\text{ in } }[-h,h]} \text{  for  some } {\Phi} \in C^1([-h,h]),~ {\Phi}>0  \right\}.
\end{aligned}
\end{equation}

Especially, let $\mathbb {Y}=\{u\in C^1([0, h])\mid u(0)=0, u>0\text{ on }(0, h]\},$ equipped with the $C^1$-norm.
Considering the following operator $\tilde {\mathscr L}: \mathbb {Y}\longrightarrow C([0, h])$ defined by
\begin{equation}
\tilde {\mathscr L}[u](x)=(\mathcal{L}_{(0, h)}^{d}+a_0-\nu\nabla) [u](x),~x\in (0, h],
\end{equation}
{ according to the  Theorem~4.1 by} Li et al. in \cite{Lifang2016}, the operator $\tilde {\mathscr L}$  admits a real  principal eigenvalue 
\begin{equation}\label{eq-303}
\lambda_p(\tilde {\mathscr L})=\sup_{\substack{ {0<u\in C^1((0, h])}}}\inf_{x \in (0, h]}\frac{\tilde{\mathscr {L}} [u](x)}{u(x)}
\end{equation}
with a positive eigenfunction $\phi_h (x)\in Y.$

Moreover, 
the following vital properties hold.

\begin{property}\label{per1}
	Assume that {\rm (\textbf{J})} holds, $a_0>0$ and $0<h<+\infty,$ 
		the 
	following properties hold$:$
	\newline	{\rm (1) } $\lambda_{p}(\tilde {\mathcal L})$ is strictly increasing and 
		continuous in 
$h\in(0,\infty);$
\newline	{\rm (2)} $\lambda_{p}(\tilde {\mathcal L})$ is strictly increasing in $a_0;$
	\newline{\rm (3)} $\lim\limits_{h \rightarrow+\infty} 
	\lambda_{p}(\tilde {\mathcal L})=a_0;$
	\newline{\rm (4)} $\limsup\limits_{h\rightarrow 0^+} 
	\lambda_{p}(\tilde {\mathcal L})\leq a_0-d$.
	\end{property}
\begin{proof}	
Motivated by Proposition 3.4 in \cite{cao2019}, the  property $(1)$ can be obtained. Clearly, we can get  the strict monotonicity of $\lambda_{p}({\mathcal{L}}_{\left(-h,  h\right)}^d+a_0-\nu\nabla )$ in $a_0$ by the definition of $(\ref{lam})$.

Now we aim to prove the asymptotic properties (3) and (4).

The proof of property (3). Denote $\mathcal D(h):= H^1\left([-{h}, {h}]\right),$ according to the variational method, $\lambda_p(\tilde{\mathcal L})$ can be expressed as 
	\begin{equation}
	\begin{array}{ll}
{\lambda}_{p}(\tilde{\mathcal{L}})&=\sup\limits _{0 \not\equiv \psi \in 
\mathcal D(h)} 
\frac{d 
\int_{-{h}}^{{h}} \int_{-{h}}^{{h}} J\left(x-y\right) \psi(y) 
\psi(x) {\rm d} y 
{\rm d} 
x}{\int_{-{h}}^{{h}} \psi^{2}(x) {\rm d} 
x}-\dfrac{\nu\int_{-{h}}^{{h}} \psi(x)\psi^{\prime}(x) 
{\rm d} 
x}{\int_{-h}^{{h}} \psi^{2}(x) {\rm d} x}-d+a_0
\\&=\sup\limits _{0\not\equiv \psi \in \mathcal D(h)} \frac{d 
	\int_{-h}^{{h}} \int_{-{h}}^{{h}} J\left(x-y\right) \psi(y) 
	\psi(x) {\rm d} 
	y 
	{\rm d} 
	x}{\int_{-{h}}^{{h}} \psi^{2}(x) {\rm d} 
	x}-\dfrac{\frac{\nu}{2}\left(\psi^2({h})-
	\psi^2(-h)\right)
	}{\int_{-{h}}^{{h}} \psi^{2}(x) {\rm d} x}-d+a_0
\\&\leq \dfrac{d 
	\int_{-{h}}^{{h}}  \psi^2(x) 
	{\rm d} 
	x}{\int_{-{h}}^{{h}} \psi^{2}(x) {\rm d} 
	x}-\dfrac{\frac{\nu}{2}\left(\psi^2({h})-
	\psi^2(-h)\right)
}{2{h} \psi^{2}(\tilde h) }-d+a_0
\\&\leq a_0+\dfrac{\frac{\nu}{2}\left(\psi^2({h})+
	\psi^2(-h)\right)
}{2{h} \psi^{2}(\tilde h) }
\\& \longrightarrow a_0, \enspace \text{as} \enspace 
h\rightarrow 
\infty,
\end{array}	
\end{equation}
where $\tilde h\in (-{h}, {h})$, then 
\begin{equation}\label{lambda_p}
\limsup\limits_{h\rightarrow\infty}
{\lambda}_{p}(\tilde{\mathcal{L}})\leq 
a_0.
\end{equation}

By $(\mathbf{J}),$ for any small $\epsilon>0,$ there exists $L=L(\epsilon)>0$ 
such that
$$
\int_{-L}^{L} J\left (x\right ){\rm d }x>1-\epsilon.
$$
Taking ${\Phi} \equiv 1$ as the test function in the variational characterization of $\lambda_{p}(\tilde{\mathcal{L}})$.
  We 
obtain
\begin{equation}
\begin{aligned}
	\lambda_{p}(\tilde{\mathcal{L}})&\geq 
	\frac{d 
	\int_{-h}^{h} \int_{-h}^{h} J\left(x-y\right)  {\rm d} y{\rm  d} x
	}{2h}-d+a_0 \\
	& \geq \frac{d \int_{-h+L}^{h-L} \int_{-h}^{h} J\left(x-y\right)  {\rm d} y{\rm  d} x
	}{2h}-d+a_0 \\
		& \geq \frac{ d(2h-2 L) \int_{-L}^{L} J\left(s\right) {\rm d}
	s}{2h}-d+a_0 \\
	&\geq \frac{d(2h-2 
	L)(1-\epsilon)}{2h}-d+a_0
	\\&\rightarrow-\epsilon d +a_0, \enspace\text{as}\enspace h\rightarrow 
	\infty.
\end{aligned}
\end{equation}
Since small $\epsilon>0$ is arbitrarily chosen, we can obtain
$$
\liminf_{h \rightarrow+\infty} \lambda_{p}(\tilde{\mathcal{L}}) \geq a_0.
$$
Combining the above inequality with $(\ref{lambda_p})$, it follows
$$
\lim_{h \rightarrow+\infty} \lambda_{p}(\tilde{\mathcal{L}})=a_0.
$$

The proof of property (4). Since $a_0$ and $\nu$ are fixed constants, $\lambda_{p}(\mathcal{L}_{(-h, 
	h)}^d+a_0-\nu\nabla)$ only depends on the integral interval $(-h, 
	h)$ from the definition $(\ref{lam})$. Without loss of generality, denote  
	$$\lambda_{p}(\mathcal{L}_{(-h, 
	h)}^d+a_0-\nu\nabla):=\lambda_{p}(\mathcal{L}_{(0, 
	\hat h)}^d+a_0-\nu\nabla).$$

According to $(\ref{eq-303})$, $\lambda_{\hat h}:=\lambda_{p}(\mathcal{L}_{(0, 
\hat h)}^d+a_0-\nu\nabla)$ is the
principal eigenvalue with an eigenfunction ${\Phi}_{\hat h}$ 
  satisfying ${\Phi}_{\hat h}(x)>0, $ $x\in (0,\hat h]$ and 
${\Phi}_{\hat h}(0)=0$ such that
$$
d \int_{0}^{\hat h} J\left(x-y\right) {\Phi}_{\hat h}(y) {\rm d} y-d~ {\Phi}_{\hat h}(x)+a_0
{\Phi}_{\hat h}(x)-\nu{\Phi}_{\hat h}^{\prime}(x)=\lambda_{\hat h} {\Phi}_{\hat h},
\text { for } x \in (0, \hat h).
$$

Therefore,
\begin{equation}
\begin{aligned}
	\lambda_{\hat h}-a_0+d&=\frac{d \int_{0}^{\hat h} \int_{0}^{\hat 
	h} 
	J\left(x-y\right) 
	{\Phi}_{\hat h}(y) {\Phi}_{\hat h}(x)  {\rm d} y{\rm  d} x }{\int_{0}^{\hat h} {\Phi}_{\hat 
	h}^{2}(x) {\rm d} 
	x}-\dfrac{\nu\int_{0}^{\hat h} 
	{\Phi}_{\hat h}(x){\Phi}_{\hat h}^{\prime}(x) 
	{\rm d} 
	x}{\int_{0}^{\hat h} {\Phi}_{\hat h}^{2}(x) {\rm d} x}
\\& \leq 
	\frac{d\|J\|_{\infty}\left(\int_{0}^{\hat h} {\Phi}_{\hat h}(x) {\rm d} 
	x\right)^{2}}{\int_{0}^{h} {\Phi}_{\hat h}^{2}(x) {\rm d }x} 
	-\dfrac{\frac{\nu}{2}\left({\Phi}_{\hat h}^2(\hat h)-{\Phi}_{\hat 
	h}^2(0)\right)}{\int_{0}^
		{\hat h}
	 {\Phi}_{\hat h}^{2}(x) {\rm d} x}
 \\
	&\leq \dfrac{ d\|J\|_{\infty} \hat h \int_{0}^{\hat h} {\Phi}_{\hat h}^{2}(x) 
	d x}
	{\int_{0}^{\hat h} {\Phi}_{\hat h}^{2}(x) {\rm d} 
	x}-\dfrac{\frac{\nu}{2}\left({\Phi}_{\hat h}^2(\hat h)-{\Phi}_{\hat 
		h}^2(0)\right)}{\int_{0}^
	{\hat h}
	{\Phi}_{\hat h}^{2}(x) {\rm d} x}\\
	&=d\|J\|_{\infty} \hat h -\dfrac{\frac{\nu}{2}\left({\Phi}_{\hat h}^2(\hat h)-{\Phi}_{\hat 
		h}^2(0)\right)}{\int_{0}^
	{\hat h}{\Phi}_{\hat h}^{2}(x) {\rm d} x}.
\end{aligned}
\end{equation}
Since
\begin{equation}\label{eq3-08}
	\dfrac{\frac{\nu}{2}\left({\Phi}_{\hat h}^2(\hat h)-{\Phi}_{\hat 
			h}^2(0)\right)}{\int_{0}^
		{\hat h}
		{\Phi}_{\hat h}^{2}(x) {\rm d} x}=
		\dfrac{\frac{\nu}{2}{\Phi}_{\hat h}^2(\hat h)}{\hat h {\Phi}_{\hat h}^2(\hat 
		h_0)}, 
\end{equation}
where $\hat h_0\in(0, \hat h ).$ According to the choice of $\Phi_{\hat h}(x)$ and $\hat h_0$, we can see that $\Phi_{\hat h}(\hat h) \geq \Phi_{\hat h}(\hat h_0)>0$ as ${\hat h\rightarrow 0}$. 
Then the right side of $(\ref{eq3-08})$ is greater than $0$.
For $\nu>0,$ it follows $\lambda_{\hat h}< a_0-d.$
Thus,  $$\limsup\limits_{h \rightarrow 0^+} 
\lambda_{p}(\mathcal{L}_{(-h, 
	h)}^d+a_0-\nu\nabla)\leq a_0-d.$$
\end{proof}
\begin{rem}
As we can see that  when the advection rate $\nu>0, ~\lambda_{p}(\mathcal{L}_{\left(-h, 
	h\right)}^d+a_0-\nu\nabla)$  can be strictly less than $a_0-d$ as $h$ is taken small enough. This implies the nonnegligible effect of the advection on the principal eigenvalue compared with the case without the advection.
\end{rem}

\section{Long-time asymptotic behavior }\label{s4}
Given $(\ref{eq2-13})$, according to the strict monotonicity of $g(t)$ and $h(t)$  in $t$,
denote \[g_\infty:=\lim_{t\rightarrow\infty}g(t), \enspace
h_\infty:=\lim_{t\rightarrow\infty}h(t),\]
then $g_\infty\in\left[-\infty, -h_0\right), \text { and }  
h_\infty\in\left(h_0,\infty\right].$
\begin{definition}
	 The  {vanishing} happens if $$h_\infty-g_\infty<\infty \text{ and }
	\limsup\limits_{t\rightarrow \infty}u\left(t,x\right)=0;$$
	the   {spreading} happens  if $$h_\infty-g_\infty=\infty \text{
	and }
	\liminf\limits_{t\rightarrow \infty}u(t,x)>0.$$
\end{definition}
For $-\infty< h_1<h_{2}<+\infty,$ denote $\mathscr D:=\left(h_1, h_{2}\right).$ Considering the following problem over $\mathscr D$
\begin{equation}\label{eq5-1}
\begin{cases}
u_{t}=d\int_{ h_1}^{h_{2}} J\left(x-y\right) u(t, y) {\rm d} y-d ~u(t, x)-\nu u_x+f(u), & t>0,~x \in\mathscr D ,
\\ u(0, x)=u_{0}(x), & x \in \mathscr D,   
\end{cases}
\end{equation}

According to the arguments in \cite{bates2007} and \cite{Coville2010},  it implies
\begin{prop}\label{pro5.1}
Suppose that $(\mathbf{J})$ and $(\mathbf{f 1})-(\mathbf{f 4})$ hold. The problem $(\ref{eq5-1})$ has a unique positive steady state $u_{\mathscr D}$ in $C^1(\bar{\mathscr D})$ if and only if
$$
\lambda_{p}\left(\mathcal{L}_{\mathscr D}^d+f_0-\nu\nabla\right)>0.
$$
Moreover, for $u_{0}(x) \in C^1(\bar{\mathscr D})$ and $u_{0}$ nonnegative and not always equal to $0$, the problem   $(\ref{eq5-1})$ admits a unique solution $u(t, x)$  for all $t>0$, and $u(t,x)\rightarrow u_{\mathscr D}$ in $C(\bar{\mathscr D})$ as $t \rightarrow+\infty$ when $\lambda_{p}\left(\mathcal{L}_{\mathscr D}^d+f_0-\nu\nabla\right)>0;$ when $\lambda_{p}\left(\mathcal{L}_{\mathscr D}^d+f_0-\nu\nabla\right) \leq 0,$ $ u(t, x)\rightarrow 0$
 in $C(\bar{\mathscr D})$ as $t \rightarrow+\infty$.
\end{prop}

\begin{theorem}\label{th5.1}
Assume that $(\mathbf{J})$ and $(\mathbf{f1})-(\mathbf{f4})$ hold. If $h_\infty- g_\infty<\infty$, then 
\begin{equation}\label{eq4-01}
\lambda_p(\mathcal L^{d}_{(g_\infty,h_\infty)}+f_0-\nu\nabla)\leq 0
\end{equation}
 and
$\lim\limits_{t\rightarrow\infty}u(t, x)=0$ uniformly in $[g(t),h(t)]$. 
\end{theorem}
\begin{proof}
Now we divide our proof into two steps.

Step~1:  We aim to show
$$
\lambda_{p}(\mathcal{L}_{(g_{\infty}, h_{\infty})}^d+f_0-\nu\nabla) \leq0 .
$$
Suppose that $\lambda_{p}(\mathcal{L}_{(g_{\infty}, h_{\infty})}^d+f_0-\nu\nabla)>0$,  according to Property~\ref{per1}, there is small $\epsilon_0>0$ 
such that  
\begin{equation}\label{eq4.3}
\lambda_{p}(\mathcal{L}_{(g_{\infty}+\epsilon, h_{\infty}-\epsilon)}^d+f_0-\nu\nabla)>0
\end{equation}
for  $\epsilon\in(0, \epsilon_0)$. 
Further, for the above $\epsilon$, there is $T_{\epsilon}>0$ such that
$$
h(t)>h_{\infty}-\epsilon, ~ g(t)<g_{\infty}+\epsilon, \text { for } t>T_{\epsilon} .
$$
Considering the following problem
\begin{equation}\label{eq5-2}
	\begin{cases}w_{t}=d \int_{g_{\infty}+\epsilon}^{h_{\infty}-\epsilon} J\left(x-y\right) w(t, y) {\rm d} y-d ~w-\nu w_x+f(w), & t>T_{\epsilon}, ~x \in\left[g_{\infty}+\epsilon, h_{\infty}-\epsilon\right], \\ w\left(T_{\epsilon}, x\right)=u\left(T_{\epsilon}, x\right), & x \in\left[g_{\infty}+\epsilon, h_{\infty}-\epsilon\right],
\end{cases}
\end{equation}
according to the assumption of $(\ref{eq4.3})$, by Proposition $\ref{pro5.1},$ the solution $w_{\epsilon}(t, x)$ of $(\ref{eq5-2})$ converges to the unique steady state $w_{\epsilon}(x)$ of $(\ref{eq5-2})$ uniformly in $\left[g_{\infty}+\epsilon, h_{\infty}-\epsilon\right]$ as $t \rightarrow+\infty$.
Further, by {the Maximum principle} and comparison argument, it follows 
$$
u(t, x) \geq w_{\epsilon}(t, x), \text { for } t>T_{\epsilon}\text { and } x \in\left[g_{\infty}+\epsilon, h_{\infty}-\epsilon\right].
$$
Then we can find a positive $\tilde T_{ \epsilon}>T_{\epsilon}$ such that
$$
u(t, x) \geq   \dfrac{1}{2}  w_{\epsilon}(x)>0,  \text { for } t>\tilde T_{ \epsilon} \text { and } x \in\left[g_{\infty}+\epsilon, h_{\infty}-\epsilon\right].
$$
By $(\mathbf{J})$, there exist $\tilde \epsilon>0$ and $\delta_{0}>0$ such that $J(x)>\delta_{0}$ for $x\in(-\tilde \epsilon,\tilde \epsilon)$. Then for   $0<\epsilon<\min \left\{\epsilon_{0}, \tilde \epsilon /2 \right\}$, we have
\begin{equation}\label{eq4.5}
\begin{aligned}
h^{\prime}(t)&=\mu \int_{g(t)}^{h(t)} \int_{h(t)}^{+\infty} J\left(x-y\right) u(t, x)  {\rm d} y{\rm  d} x
\\&\geq \mu \int_{g_{\infty}+\epsilon}^{h_{\infty}-\epsilon} \int_{h_{\infty}}^{+\infty} J\left(x-y\right) u(t, x)  {\rm d} y{\rm  d} x\\
&\geq \dfrac{1}{2} \mu \int_{h_{\infty}-\tilde \epsilon /2}^{h_{\infty}-\epsilon} \int_{h_{\infty}}^{h_{\infty}+\tilde \epsilon /2} \delta_{0} w_{\epsilon}(x)   {\rm d} y{\rm  d} x
\\&>0, 
\end{aligned}
\end{equation}
$\text{ for all  }t>\tilde T_{ \epsilon},$ which implies $h_{\infty}=+\infty$, contradicting the assumption that $h_{\infty}-g_{\infty}<+\infty$. Thus, it follows
$$
\lambda_{p}(\mathcal{L}_{\left(g_{\infty}, h_{\infty}\right)}^d+f_0-\nu\nabla ) \leq 0.
$$

Step~2: It turns to show that $u(t, x)$ converges to $0$ uniformly in $[g(t), h(t)]$ as $t \rightarrow+\infty$. 

Let $\bar{u}(t, x)$ be the solution of the following equation
$$
\begin{cases}
\bar{u}_{t}=d \int_{g_{\infty}}^{h_{\infty}} J\left(x-y\right) \bar{u}(t, y) {\rm d} y-d~ \bar{u}(t, x)-\nu \bar u_x+f(\bar{u}), & t>0,~x \in\left[g_{\infty}, h_{\infty}\right], \\ \bar{u}(0, x)=v_{0}(x), & x \in\left[g_{\infty}, h_{\infty}\right],
\end{cases}
$$
where
$$
v_0(x) =\begin{cases}
u_{0}(x),& \text { if }-h_{0} \leq x \leq h_{0},
\\ 0, & \text { if } x>h_{0} \text { or } x<-h_{0}.
\end{cases}
$$
By{ Maximum principle}, we can obtain that $$u(t, x) \leq \bar{u}(t, x), \text{ for } t>0 \text{ and } x \in[g(t), h(t)].$$
Furthermore, when
$
\lambda_{p}(\mathcal{L}_{(g_{\infty}, h_{\infty})}^d+f_0-\nu\nabla) \leq 0,
$
by Proposition $\ref{pro5.1}$, it follows that $\bar{u}(t, x) \rightarrow 0$ uniformly in $x \in[g_{\infty}, h_{\infty}]$ as $t \rightarrow+\infty$. Hence, $\lim\limits_{t\rightarrow\infty} u(t, x) = 0$ uniformly in $x \in[g(t), h(t)]$. 
\end{proof}

Moreover, as the spreading happens, we discuss the long-time asymptotic behaviors of the solution for system $(\ref{eq1-1})$.
\begin{lemma}\label{lem5.2}
	Let $\left(u, g, h\right)$ be the unique solution of $(\ref{eq1-1})$,
	 if $h_{\infty}-g_{\infty}=\infty$, then
	$g_{\infty}=-\infty$ and $h_{\infty}=\infty$
\end{lemma}
\begin{proof}
It suffices to show $h_{\infty}<+\infty$ if and only if $-g_{\infty}<+\infty$.
 Without loss of generality, we assume by contradiction that $h_{\infty}<+\infty$ and $g_{\infty}=-\infty$. According to Property~\ref{per1},  there exists $l_{*}>0$ such that $$\lambda_{p}(\mathcal{L}_{(-l_*, h_0)}^d+f_0-\nu\nabla)>0.$$
 Further, for any $0<\epsilon\ll1$ small, there exists $T_{\epsilon}>0$ such that $$h(t)>h_\infty-\epsilon>h_0,~ g(t)<-l_*,$$
 for $t>T_{\epsilon}$. Especially,
$$
\lambda_{p}(\mathcal{L}_{(-l_*, h_\infty-\epsilon)}^d+f_0-\nu\nabla)>\lambda_{p}(\mathcal{L}_{(-l_{*}, h_0)}^d+f_0-\nu\nabla)>0.
$$
Consider
\begin{equation}
\begin{cases}
	w_{t}=d \int_{-l_*}^{h_\infty-\epsilon} J\left(x-y\right) w(t, y) {\rm d} y-d~ w-\nu w_x+f(w), & t>T_{\epsilon}, ~x \in\left[-l_{*}, h_\infty-\epsilon \right], 
	\\ w\left(T_{\epsilon}, x\right)=u\left(T_{\epsilon}, x\right), & x \in\left[-l_{*}, h_\infty-\epsilon \right],
\end{cases}
\end{equation}
 applying the similar method in proving $(\ref{eq4.5})$ of the proof for Theorem $\ref{th5.1}$, it  can be  shown that $h^{\prime}(t)>0$ for all $t$ large enough, which contradicts $h_{\infty}<+\infty$. Thus, the theorem has been proved.
\end{proof}

Assuming that $(\mathbf{f} \mathbf{1}) \mathbf{- ( \mathbf { f }} \mathbf{4})$ hold,   it can be seen that there is  unique $\tilde u_{0} \in$ $\left(0, k_{0}\right)$ such that $f(\tilde u_0)=0$. Motivated by Proposition 3.6 in~\cite{cao2019},  we have the following result.

\begin{prop} \label{pro5.2}
	Assume that $(\mathbf{J})$ and $(\mathbf{f} \mathbf{1}) \mathbf{- ( \mathbf { f }} \mathbf{4})$  hold. Then there exists $L>0$ such that for every interval $(h_1, h_{2})$ with $h_{2}- h_1>L$, 
	we have $\lambda_{p}(\mathcal{L}_{( h_1, h_{2})}^d+f_0-\nu\nabla)>0$ and hence $(\ref{eq5-1})$ admits a unique positive steady state $u_{( h_1, h_{2})}$. Moreover,
\begin{equation}\label{eq5-5}
	\lim _{ h_2 -h_1 \rightarrow+\infty} u_{\left( h_1,  h_{2}\right)}=\tilde u_{0}  \text { locally uniformly in } \mathbb{R}.
\end{equation}
\end{prop}
Further,
\begin{theorem}[Asymptotic limit]\label{ths5-6.2} Let $\left(u, g, h\right)$ 
be the 
	unique solution of $(\ref{eq1-1})$, if $h_\infty-g_\infty=\infty,$ 
	then $\lim\limits_{t\rightarrow\infty}u(t,x)=\tilde u_0$ locally uniformly in $\mathbb R$.
\end{theorem}
\begin{proof}
	According to Lemma~$\ref{lem5.2}$, $h_{\infty}-g_{\infty}=+\infty$ implies that $h_{\infty}$ and $-g_{\infty}$ are equal to $\infty$ simultaneously. Take an increasing sequence $\left\{t_{n}\right\}_{n \geq 1}$ satisfying
	$ t_{n}\rightarrow \infty$ as $n\rightarrow \infty$ and $$\lambda_{p}(\mathcal{L}_{(g(t_{n}), h(t_{n}))}^d+f_0-\nu)>0, 
	$$
$\text { for all } n \geq 1.$	Let $\underline{u}_{n}(t, x)$ be the unique solution of the following problem
	\begin{equation}\label{eq5-8}
	\begin{cases}
		\underline{u}_{t}=d \int_{g(t_n)}^{h(t_n)} J\left(x-y\right) \underline{u}(t, y) {\rm d }y-d~ \underline{u}(t, x)-\nu \underline {u}_x+f(\underline{u}), & t>t_{n}, ~x \in\left[g(t_n), h(t_n)\right], \\ \underline{u}\left(t_{n}, x\right)=u\left(t_{n}, x\right), & x \in\left[g(t_n), h(t_n)\right].
	\end{cases}
	\end{equation}
	By Proposition~$\ref{pro5.1}$ and the comparison argument, it follows that
	\begin{equation}\label{eq5-08}
	u(t, x) \geq \underline{u}_{n}(t, x) \text { in }\left[t_{n},+\infty\right) \times\left[g(t_n), h(t_n)\right].
	\end{equation}
	Since $\lambda_{p}(\mathcal{L}_{[g(t_n), h(t_n)]}^d+f_0-\nu\nabla)>0$, by Proposition $\ref{pro5.2}$, the problem $(\ref{eq5-8})$ admits a unique positive steady state $\underline{u}_{n}(x)$ and
	\begin{equation}\label{eq5-9}
	\lim_{t \rightarrow+\infty} \underline{u}_{n}(t, x)=\underline{u}_{n}(x) \text { uniformly in }\left[g(t_n), h(t_n)\right].
	\end{equation}
	By Proposition $\ref{pro5.2}$,
		\begin{equation}\label{eq5-10}
	\lim _{n \rightarrow \infty} \underline{u}_{n}(x)=\tilde u_{0} \text { uniformly for } x \text{ in any compact subset of } \mathbb{R}.
	\end{equation}
	According to $(\ref{eq5-08})-(\ref{eq5-10})$, we can get
	\begin{equation}
	\liminf _{t \rightarrow+\infty} u(t, x) \geq \tilde u_{0} \text { uniformly for } x \text{ in any compact subset of } \mathbb{R}.
	\end{equation}
Now it remains to prove that
	\begin{equation}\label{eq5-12}
	\limsup _{t \rightarrow+\infty} u(t, x) \leq \tilde u_{0}\text { uniformly for } x \text{ in any compact subset of } \mathbb{R}.
	\end{equation}
	Let $\overline{u}(t)$ be the solution of the following ordinary differential equation 
	$$
	\begin{cases}
	\overline {u}^{\prime}(t)=f(\overline {u}),&t>0,\\
	\overline{u}(0)=\left\|u_{0}\right\|_{L^\infty},
	 \end{cases}
	$$
	by the comparison principle,  it follows $$u(t, x) \leq \overline{u}(t), \text{ for } t>0 \text{ and } x \in[g(t), h(t)].$$
	Note that  $\overline {u}(t)$ converges to $\tilde u_{0}$ as $t \rightarrow \infty,$ then $(\ref{eq5-12})$ is proved.  Hence, we have completed the proof of this theorem.
\end{proof}

According to the monotonicity and continuity of $\lambda_{p}(\mathcal{L}_{(-h, 
h)}^{d}+f_0-\nu\nabla)$ in $h$ from Property~$\ref{per1}$,  when
\begin{equation}\label{s5-010}
f_0<d \enspace 
\end{equation}
holds,  
if $h$ is small enough, then
$$\lambda_{p}(\mathcal{L}_{(-h, 
h)}^{d}+f_0-\nu\nabla)<0;$$ 
if $h$ is large enough,  then $$\lambda_{p}(\mathcal{L}_{(-h, 
h)}^{d}+f_0-\nu\nabla)>0.$$
Hence, there is constant $h^{*}>0$ such that
\begin{equation}\label{s5-5.12}
\lambda_{p}(\mathcal{L}_{(-h^{*}, 
h^{*})}^{d}+f_0-\nu\nabla)=0.
\end{equation}
and $\lambda_{p}(\mathcal{L}_{(-h, 
h)}^{d}+f_0-\nu\nabla)>0,$ for $h>h^*$.

In the next discussions, we will further find out the critical conditions determining the vanishing or spreading for $(\ref{eq1-1})$. Assume that $(\ref{s5-010})$ always holds. We have the following results.
\begin{theorem}\label{th4.3}
Assume that $(\mathbf{J})$ and $(\mathbf{f} \mathbf{1}) \mathbf{- ( \mathbf { f }} \mathbf{4})$  hold. If $h_{0} \geq h^{*} $ then spreading always occurs for $(\ref{eq1-1})$. If $h_{0}<h^{*} $, then there exists $\tilde {\mu}_*>0$ such that vanishing occurs for $(\ref{eq1-1})$ if $0<\mu \leq \tilde {\mu}_*.$
\end{theorem} 

\begin{proof}
when $h_{0} \geq h^{*} $, assume by contradiction that  vanishing occurs, it yields  $-\infty<g_{\infty}< h_{\infty}<\infty \text{ and } h_\infty-g_\infty> 2h^{*}.$
According to Property~$\ref{per1}$ and $(\ref{s5-5.12})$, $$\lambda_{p}(\mathcal{L}_{\left(g_{\infty}, h_{\infty}\right)}^d+f_0-\nu\nabla)>0,$$ contradicting the Theorem \ref{th5.1}. Thus, if $h_{0} \geq h^{*}$, spreading always happens for (\ref{eq1-1}).

Now we mainly consider the case of $h_0<h
^*.$  Fix $\tilde h \in\left(h_{0}, h^{*}\right)$,  let $\omega (t, x)$ be unique solution of the following problem
\begin{equation}
\begin{cases}w_{t}(t, x)=d \int_{-\tilde h}^{\tilde h} J\left(x-y\right)w(t, y){\rm  d} y-d~ w-\nu w_x+f(w), & t>0, ~x \in[-\tilde h, \tilde h], \\ w(0, x)=u_{0}(x), & x \in\left[-h_{0}, h_{0}\right], \\ w(0, x)=0, & x \in[-\tilde h,-h_{0}) \cup(h_{0}, \tilde h],
\\w(t, \tilde h)=0,~w(t,-\tilde h)=0,&t>0.
\end{cases}
\end{equation}
The choice of $\tilde h$ makes
$$
\lambda:=\lambda_{p}(\mathcal{L}_{(-\tilde h, \tilde h)}^d+f_0-\nu\nabla)<0.
$$
Let ${\Phi}>0$ be an eigenfunction of $\lambda$  with 
 $\left\|{\Phi}\right\|_{L^\infty}=1$, then
$$
(\mathcal{L}_{(-\tilde h, \tilde h)}^d+f_0-\nu\nabla)   [{\Phi}](x)=\lambda {\Phi}(x), \text { for } x \in[-\tilde h, \tilde h].
$$
By $(\textbf{f4})$, we have 
\begin{equation}\label{eq4-12}
\begin{aligned}
	\omega_{t}(t, x) &=d \int_{-\tilde h}^{\tilde h} J (x-y) \omega (t, y){\rm  d } y-d ~\omega-\nu\omega_x+f(\omega) \\
	& \leq d \int_{-\tilde h}^{\tilde h} J(x-y) \omega(t, y) {\rm d} y-d ~\omega-\nu\omega_x+f_0 ~\omega.
\end{aligned}
\end{equation}
For given $\gamma>0$ which will be determined later,  and chosen $\tilde w=\gamma e^{\lambda t/2}{\Phi}$, we can obtain that
$$
\begin{aligned}
	& d \int_{-\tilde h}^{\tilde h} J\left(x-y\right) \tilde w(t, y) {\rm d} y-d~ \tilde w-\nu \tilde w_{x}+f_0~ \tilde w-\tilde w_{t}(t, x) \\
	=& \gamma e^{\lambda t/2}\left\{d \int_{-\tilde h}^{\tilde h} J\left(x-y\right){\Phi}(y) {\rm d} y-d ~{\Phi}-\nu{\Phi}^{\prime}+f_0 ~{\Phi}-\frac{\lambda}{2} {\Phi}\right\} \\
	=& \frac{ \lambda}{2} \gamma e^{\lambda t/2} {\Phi}<0 .
\end{aligned}
$$
Take $\gamma>0$ large such that $\gamma~{\Phi}>u_{0}$ in $[-\tilde h, \tilde h]$. Applying the Maximum principle to $\tilde w-\omega,$ we can get that
\begin{equation}\label{eq4-13}
\omega(t, x) \leq \tilde w(t, x)=\gamma e^{\lambda t/2} {\Phi} \leq \gamma e^{\lambda t/2}, \text{ for } t>0 \text{ and } x \in[-\tilde h, \tilde h].
\end{equation}

{Denote}
$$
\hat {h}(t)=h_{0}+2 \mu \tilde h \gamma \int_{0}^{t} e^{\lambda s/2} {\rm d} s, \text { for } t \geq 0,
$$
and $\hat{g}(t)=-\hat{h}(t),~t\geq 0.$ 

Next, we show that $(\omega, \hat{h}, \hat{g})$ is an upper solution of $(\ref{eq1-1})$.

Take
$$
0<\mu \leq \tilde {\mu}_*:=\frac{\lambda(h_{0}-\tilde h)}{4 \tilde h \gamma}, 
$$ 
since $\lambda <0,$ for any $t>0,$ 
we have
$$
\hat{h}(t)\leq h_{0}-\frac{4 \mu \tilde h \gamma }{\lambda}(1-e^{\lambda t/2})<h_{0}- \frac{4 \mu \tilde h \gamma }{\lambda} \leq \tilde h.
$$

Similarly, $-\tilde h<\hat{g}(t)<0,$ for any $t>0$. Thus $(\ref{eq4-12})$ gives
$$
\omega_{t}(t, x) \geq d \int_{\hat{g}(t)}^{\hat{h}(t)} J\left(x-y\right) \omega (t, y) {\rm d }y-d ~\omega-\nu\omega_x+f(\omega),  \text { for } t>0, ~x \in [\hat{g}(t), \hat{h}(t)].
$$
Due to $(\ref{eq4-13})$, it can be obtained that
$$
\int_{\hat{g}(t)}^{\hat{h}(t)} \int_{\hat{h}(t)}^{+\infty}J\left(x-y\right) \omega(t, x)  {\rm d} y{\rm  d} x<2 \tilde h \gamma  e^{\lambda t /2}.
$$
{Thus,}
$$
\hat{h}^{\prime}(t)=2 \mu \tilde h \gamma e^{\lambda t /2}>\mu \int_{\hat{g}(t)}^{\hat{h}(t)} \int_{\hat{h}(t)}^{+\infty} J\left(x-y\right) \omega (t, x) {\rm d} y {\rm d}x.
$$
And
$$
\hat{g}^{\prime}(t)<-\mu \int_{\hat{g}(t)}^{\hat{h}(t)} \int_{-\infty}^{\hat{g}(t)} J\left(x-y\right) \omega (t, x) {\rm d} y{\rm  d} x.
$$
Thus $(\omega, \hat{h}, \hat{g})$ is an upper solution of $(\ref{eq1-1})$.  By comparison principle, it follows
$$
u(t, x) \leq \hat{w}(t, x), ~g(t) \geq \hat{g}(t) \text { and } h(t) \leq \hat{h}(t), \text { for } t>0, ~x \in[g(t), h(t)].
$$
Thus,
$$
h_{\infty}-g_{\infty} \leq \limsup_{t \rightarrow+\infty} (\hat{h}(t)-\hat{g}(t)) \leq 2 \tilde h<+\infty.
$$

\end{proof}

\begin{theorem}\label{th4.4}
Assume that $(\mathbf{J})$ and $(\mathbf{f} \mathbf{1}) \mathbf{- ( \mathbf { f }} \mathbf{4})$  hold.	 If $h_{0}<h^{*} $, then there exists $\tilde{\mu}^*>0$ such that spreading happens to $(\ref{eq1-1})$ if $\mu>\tilde {\mu}^*$.
	\end{theorem}

\begin{proof}
 Suppose that $ h_{\infty}-g_{\infty}<+\infty$ for any $\mu>0,$ we will give a proof by contradiction.

According to Theorem $\ref{th5.1}$, we have $\lambda_{p}(\mathcal{L}_{(g_{\infty}, h_{\infty})}^d+f_0-\nu\nabla) \leq 0$, it then follows $$h_{\infty}-g_{\infty} \leq 2h^{*}.$$ Let $\left(u_{\mu}, g_{\mu}, h_{\mu}\right)$ represent the solution of (\ref{eq1-1}) to stress the dependence on $\mu$.  By the comparison principle, we can show the strict monotonicity of $u_{\mu}, g_{\mu}$  and $h_{\mu}$ in $\mu$. 

Denote
$$
h_{\mu} (\infty):=\lim _{t \rightarrow+\infty} h_{\mu}(t)\text{ and } g_{\mu}(\infty):=\lim _{t \rightarrow+\infty} g_{\mu}(t) .
$$
Then $h_{\mu} (\infty)$ is increasing in $\mu$ and $g_{\mu} (\infty)$ is decreasing in $\mu$. 

Denote
$$
\mathscr H_{\infty}:=\lim _{\mu \rightarrow+\infty} h_{\mu} (\infty)\in(0,\infty] \text{ and } \mathscr G_{\infty}:=\lim _{\mu \rightarrow+\infty} g_{\mu} (\infty) \in[-\infty,0).
$$
Since $J(0)>0$, there exist $0<\epsilon_{0}<h_0$ and $\delta_{0}>0$ such that $$J(x)\geq \delta_{0}, \text{ for }x\in(-\epsilon_{0},\epsilon_{0}).$$  And there exist  constants  $\tilde \mu$ and $ \tilde t$ such that  $$h_{\mu}(t)>\mathscr H_{\infty}-\frac{\epsilon_{0}}{2},\text{ and }  g_{\mu}(t)<\mathscr G_{\infty}+\frac{\epsilon_{0}}{2}$$ 
for any $\mu \geq \tilde \mu$ and $t \geq \tilde t$. Then, by $(\ref{eq1-1})$, it follows 
\begin{equation}
	\begin{aligned}
	&\mu=\dfrac{h_{\mu}(\infty)-h_{\mu}\left(\tilde t\right)} {\int_{\tilde t}^{+\infty} \int_{g_{\mu}(\tau)}^{h_{\mu}(\tau)} \int_{h_{\mu}(\tau)}^{+\infty} J\left(x-y\right) u_{\mu}(\tau, x) {\rm d } y {\rm d}  x {\rm d} \tau}
	\\&\leq \dfrac{2h^*}{\int_{\tilde t}^{2\tilde t} \int_{g_{\tilde \mu}(\tau)}^{h_{\tilde\mu}(\tau)} \int_{h_{\tilde \mu}(\tau)+\frac{\epsilon_{0}}{2}}^{+\infty} J\left(x-y\right) u_{\tilde\mu}(\tau, x) {\rm d } y {\rm d}  x {\rm d}\tau}     \\
		&\leq \dfrac{2h^{*}}{\delta_{0} \int_{\tilde t}^{2\tilde t} \int_{h_{\tilde \mu}(\tau)-\epsilon_{0}}^{h_{\tilde\mu}(\tau)} \int_{h_{\tilde\mu}(\tau)+\frac{\epsilon_{0}}{2}}^{h_{\tilde \mu}(\tau)+\epsilon_{0}} u_{\tilde \mu}(\tau, x) {\rm d } y {\rm d}  x {\rm d} \tau}
		 \\&=\dfrac{2h^{*}}{\frac{1}{2} \delta_{0} \epsilon_{0} \int_{\tilde t}^{2\tilde t} \int_{h_{\tilde \mu}(\tau)-\epsilon_{0}}^{h_{\tilde \mu}(\tau)} u_{\tilde \mu}(\tau, x) 
		 {\rm d } x {\rm d} \tau}
		\\&<+\infty.
	\end{aligned}
\end{equation}
Since $\mu$ can be taken large enough, it yields a contradiction.
Therefore, we can take
$$
\tilde{\mu}^*=1+\dfrac{4h^{*}}{\delta_{0} \epsilon_{0} \int_{\tilde t }^{2\tilde t} \int_{h_{\tilde \mu}(\tau)-\epsilon_{0}}^{h_{\tilde \mu}(\tau)} u_{\tilde \mu}(\tau, x)
 {\rm d} x {\rm d} \tau},
$$
the proof is completed.
\end{proof}
Now we can give a more explicit dichotomy description of $\mu$ as follows.
\begin{theorem}
Assume that $(\mathbf{J})$ and $(\mathbf{f} \mathbf{1}) \mathbf{- ( \mathbf { f }} \mathbf{4})$  hold. when $h_{0}<h^{*},$  there exists $\mu^{*}>0$ such that vanishing happens for $(\ref{eq1-1})$ if $0<\mu \leq \mu^{*}$ and spreading happens for $(\ref{eq1-1})$ if $\mu>\mu^{*}$.
\end{theorem}
\begin{proof}
Denote $\mu^{*}:=\sup \Omega,$ where 
$
\Omega=\left\{\mu\mid \mu>0 \text { such that } h_{\infty}-g_{\infty}<+\infty\right\} .
$
Given Theorems \ref{th4.3} and \ref{th4.4}, we see that $0<\mu^*<+\infty$. Let $\left(u_{\mu}, g_{\mu}, h_{\mu}\right)$ be the solution of $(\ref{eq1-1})$ and set $$h_{\mu, \infty}:=\lim\limits_{t \rightarrow+\infty} h_{\mu}(t) \text{ and } g_{\mu, \infty}:=\lim \limits_{t \rightarrow+\infty} g_{\mu}(t).$$
Since $u_{\mu}, -g_{\mu}$ and $h_{\mu}$ are increasing in $\mu$. It can be obtained  that if $\mu_{0} \in \Sigma$, then $\mu \in \Omega$ for any $\mu<\mu_{0}$ and if $\mu_{0} \notin \Omega$, then $\mu \notin \Omega$ for any $\mu>\mu_{0}$. Thus,
\begin{equation}\label{eq4-15}
\left(0, \mu^{*}\right) \subseteq \Omega\text{ and } \left(\mu^{*},+\infty\right) \cap \Omega=\emptyset.
\end{equation}

Next we show that $\mu^{*} \in \Omega$ by contradiction. Suppose that $\mu^{*} \notin \Omega$. Then $h_{\mu^{*}, \infty}=-g_{\mu^{*}, \infty}=+\infty$. Thus there is $T>0$ such that $-g_{\mu^{*}}(t)>h^{*},$ $h_{\mu^{*}}(t)>h^{*}$ for $t \geq T$. Hence there exists $\epsilon>0$ such that  $-g_{\mu}(T)>h^{*}$ and 
$ h_{\mu}(T)>h^{*}$ for $\mu \in \left(\mu^*-\epsilon, \mu^*+\epsilon\right)$, which implies $\mu \notin \Omega$.  It contradicts $(\ref{eq4-15}).$ Therefore $\mu^{*} \in \Omega$.
\end{proof}

Given the above theorems, we 
can get the following spreading-vanishing theorems.

\begin{theorem}[Spreading-vanishing criteria] Assume that $J$
and $f$ satisfy the conditions of Theorem $\ref{th5.1}$.  Let 
	$\left(u, g, h\right)$ be the unique solution of $(\ref{eq1-1}),$  if $f_0<d,$
 then there exists a 
unique 
	$h^{*}>0$ such that
\newline{	{\rm (1)}} vanishing occurs and $h_{\infty}-g_{\infty} 
\leq 
	2h^{*};$
\newline{	{\rm (2)}} spreading  occurs when $h_{0} \geq h^{*};$
\newline{	{\rm (3)}} if $h_{0}<h^{*},$ then there exists a positive 
constant
	$\mu^{*}  >0$ such that vanishing occurs when 
	$\mu \leq \mu^{*}$ and spreading occurs when 
	$\mu>\mu^{*}.$
	\end{theorem}

Furthermore, we also have the following spreading and vanishing dichotomy regimes.
\begin{theorem}[Spreading-vanishing dichotomy]
Let $(u, g, h)$  be the solution of  $(\ref{eq1-1})$, then one of the following regimes hold for $(\ref{eq1-1}):$
 \newline {\rm (1)} {vanishing}$:$ $h_\infty-g_\infty<\infty$ and $\lim\limits_{t\rightarrow\infty}u(t, x)=0$ uniformly in $[g(t),h(t)];$
 \newline {\rm (2)} {spreading}$:$ $h_\infty-g_\infty=\infty$ and 
	$\lim\limits_{t\rightarrow\infty}u(t,x)=\tilde u_0$ locally uniformly in $\mathbb R$.
\end{theorem}

\section{Spreading speed}\label{s5}
In this section, we will mainly investigate the effects of the advection on the spreading speeds of double free boundaries for the problem $(\ref{eq1-1})$.
{We aim to find out the explicit differences between the leftward and 
rightward 
asymptotic spreading speeds induced by the advection term.} 

For the sake of simplicity, without loss of generality, we take $f(u):=u(1-u)$ in problem $(\ref{eq1-1})$ in the next discussions. Our results can be applied to other Fisher-KPP type or monostable type models.

The following theorem is our main result in this section which represents the propagation speed of the leftward front is strictly less than the rightward front for $\nu>0$. Simultaneously, it shows that the spreading speed for the $(\ref{eq1-1})$ is finite if and only if the following assumption is satisfied:
\newline $\left(\mathbf{J_*}\right)$  $\int_{0}^{\infty}x J(x)dx<\infty.$

\begin{thm}\label{th1-1}
	Assume that {\rm($\textbf{J}$)} and $(\mathbf{f} \mathbf{1}) \mathbf{- ( \mathbf { f }} \mathbf{4})$  hold. If {\rm($\mathbf{J_*}$)} is also satisfied, when the 
	spreading occurs, the asymptotic spreading speeds of the leftward front 
	and the rightward front for the problem $(\ref{eq1-1})$ satisfy
	\begin{equation}
	\begin{array}{ll}
			\lim\limits_{t\rightarrow\infty}\dfrac{g(t)}{t}=c_{l}^*,
			~\lim\limits_{t\rightarrow\infty}\dfrac{h(t)}{t}=c_{r}^*.
				 \end{array}
	\end{equation} 
Moreover,
\begin{equation}
	0<c_{l}^*<c^*<c_{r}^*,~\lim\limits_{\nu\rightarrow 
	0}c_{l}^*=\lim\limits_{\nu\rightarrow 0}c_{r}^*=c^*.
	\end{equation}
	where 
$c^*$ is exactly the finite asymptotic spreading speed of the double free boundaries for the problem $(\ref{eq1-1})$ without the advection.
\end{thm}
However,  under the assumptions of the above theorem,
\begin{thm}\label{th7.8}
	If $\left(\mathbf{J_*}\right)$ is not satisfied, as the spreading occurs, it follows
	\[\lim_{t\rightarrow 
		\infty}\dfrac{h(t)}{t}=\lim_{t\rightarrow 
		\infty}\dfrac{-g(t)}{t}=\infty.\]
\end{thm}

For 
\begin{equation}\label{eq1-2-2}
	\begin{cases}
		u_{t}=d \int_{\mathbb{R}} J\left(x-y\right) u(t, y){\rm d} y-d ~u(t, x)-\nu 
		u_x+f(u), & t>0,~x 
		\in \mathbb{R}, 
		\\ u(0, x)=u_{0}(x), & x \in \mathbb{R},
	\end{cases}
\end{equation}
assume that the kernel function $J(x)$ satisfies
\newline$\left(\mathbf{J_{**}}\right)$  $\int_{-\infty}^{\infty}e^{\lambda x} J(x)dx<\infty, \text{ for~some } \lambda>0.$ 

It can be easily shown that $\left(\mathbf{J_{**}}\right)$ implies $\left(\mathbf{J_{*}}\right)$, while the reverse is not true. For instance, $J(x)=(1+|x|)^{-3}$ satisfies 
$\left(\mathbf{J_{*}}\right)$ but does not satisfy $\left(\mathbf{J_{**}}\right)$.

According to Theorem~$\mathbf{1}$ in~\cite{yagisita2010existence},  we can get 
the 
following 
vital result.
\begin{prop}\label{prop1-3}
	Suppose that $J$ satisfies $(\mathbf{J})$, $(\mathbf{J_*})$ and  $f$ satisfies $(\mathbf{f} \mathbf{4})$. If additionally, $J$ satisfies $(\mathbf{J_{**}})$, 
	then there 
	is a constant $\hat c_{*}>0$ such that $(\ref{eq1-2-2})$ has a traveling wave 
	solution with 
	speed $c$ if and only if $c \geq \hat c_{*}$. In fact, the following problem
	\begin{equation}\label{eq1-4-4}
	\left\{\begin{array}{l}
		d \int_{\mathbb{R}}J\left(x-y\right) \Phi(y) {\rm d }y-d ~
		\Phi(x)+(c-\nu) 
		\Phi^{\prime}(x)+f(\Phi(x))=0, ~ x \in \mathbb{R}, \\
		\Phi(-\infty)=1, ~\Phi(+\infty)=0,
	\end{array}\right.
	\end{equation}
	has a solution $\Phi \in L^{\infty}(\mathbb{R})$ which is nonincreasing if 
	and only if $c \geq \hat c_{*} .$ Moreover, for each $c \geq \hat c_{*}$, the 
	solution satisfies $: \Phi \in 
	C^{1}(\mathbb{R})$. 
	Meanwhile, if $J$ does not satisfy $(\mathbf{J_{**}})$, then 
	$(\ref{eq1-2-2})$ does 
	not have a traveling wave solution. 
\end{prop}

For the problem 
\begin{equation}\label{eq1-2002}
	\left\{\begin{array}{l}d  \int_{-\infty}^{0} J\left(x-y\right)
		\Phi(y) \mathrm{d} y-d ~ \Phi+c\Phi^{\prime}(x)+f
		(\Phi(x))=0, ~ -\infty<x<0, \\ \Phi(-\infty)=1, ~
		\Phi(0)={0},
	\end{array}\right.
\end{equation}
Assume that the kernel function $J(x)$ satisfies
$\left(\mathbf{J_{**}}\right),$ it follows
\begin{prop}[Theorem~2.6,~\cite{du2021semi}]\label{prop1-1}
	Suppose that $(\mathbf{J}),(\mathbf{J_{**}})$ and 
	$(\mathbf{f} \mathbf{4})$ 
	hold. There is a constant $\tilde c>0$ such that for any $c \in\left(0, \tilde c\right)$, the problem 
	$(\ref{eq1-2002})$ has a 
	unique solution $\Phi=\Phi_r^{c}$, and $\Phi_r^c(x)$ is strictly decreasing 
	in $c \in\left(0, \tilde c\right)$ for fixed $x<0$, and is strictly decreasing 
	in $x \in(-\infty, 0]$ for fixed $c \in\left(0, \tilde c\right)$.
\end{prop}

Consider the following problem
\begin{equation}\label{eq1-2}
	\left\{\begin{array}{l}d  \int_{-\infty}^{0}J\left(x-y\right)
		\Phi(y) \mathrm{d} y-d ~ \Phi+(c-\nu) \Phi^{\prime}(x)+f
		(\Phi(x) 
		)=0, 
		~-\infty<x<0, \\ \Phi(-\infty)=1, ~
		\Phi(0)={0},
			\end{array}\right.
\end{equation}
with 
\begin{equation}\label{eq1-202}
		c=\mu\int_{-\infty}^{0} \int_{0}^{\infty} 
		J\left(x-y\right) 
		\Phi(x) \mathrm{d} y \mathrm{d} x.
\end{equation}
To investigate the asymptotic spreading speeds of $(\ref{eq1-1})$, motivated by Theorem~2.7~\cite{du2021semi}, we first propose the following result.
\begin{lem}\label{le-05.1}
	Denote $\tilde c_r:=\tilde c+\nu$. Assume that $(\mathbf{J}),(\mathbf{J_{**}})$ and $(\mathbf{f} 
	\mathbf{4})$ hold, for any $c_r\in(0, \tilde c_r)$, the problem 
	$(\ref{eq1-2})$ admits a semi-wave solution $\Phi^{c_r}(x)$ satisfying
	\begin{equation}\label{eq1-3-3}
		\lim _{c_r \rightarrow {\tilde c_{r}}^-} \Phi^{c_r}(x)=0 \text { locally 
		uniformly 
		in } x \in(-\infty,0]. 
	\end{equation}	
Further, for any $\mu>0,$ there is  unique $c=c_r^*\in(0,\tilde c_r)$ such that 
\begin{equation}\label{eq1-303}
c_{r}^*=\mu \int_{-\infty}^{0} \int_{0}^{\infty}J\left(x-y\right)\Phi^{c_{r}^*}(x) {\rm 
	d }y 
	{\rm d }x.
\end{equation}	
\end{lem}
\begin{proof}
The proof will be completed in two steps.

Step~1: We aim to prove 	$(\ref{eq1-3-3})$.

Considering the following problem 
	\begin{equation}\label{eq1-6-6}
		\left\{\begin{array}{l}d  \int_{-\infty}^{0} J\left(x-y\right)
					\Phi(y) \mathrm{d} y-d~  \Phi+(c-\nu) \Phi^{\prime}(x)+f
			(\Phi(x) )=0,~ 
			-\infty<x<0, \\ \Phi(-\infty)=1, ~
			\Phi(0)={0},
		\end{array}\right.
	\end{equation}
according to {Proposition~$\ref{prop1-1}$}, there exists a $c_0=\tilde c+\nu$ such 
that 
$(\ref{eq1-6-6})$ admits a semi-wave solution pair $(c,\Phi^c(x))$ for any 
$c<c_0$ with $\Phi^c(x)$ strictly decreasing.  

Let $\{c_{n}\}\subset \left(0, c_{0}\right) $ be an arbitrary sequence which increasingly 
converges to $c_{0}$ as $n \rightarrow \infty$. Denote $\Phi_{n}(x):=$ 
$\Phi^{c_{n}}(x)$. We can see that $\Phi_{n}(x)$ is uniformly bounded, and  $\Phi_{n}^{\prime}(x)$ is also 
uniformly bounded in view of 
 $(\ref{eq1-6-6})$. Therefore there is a subsequence $\{\Phi_{n_k}\}$ of $\Phi_{n}$ such that 
 $$\Phi_{n_k}(x) \rightarrow \Phi(x) \text{ in } C_{l o 
c}((-\infty, 0]) \text{ as } k\rightarrow \infty .$$
Without loss of generality, we still 
denote $\Phi_{n_k}$ by $\Phi_{n}$. By 
{Proposition~$\ref{prop1-1}$}, the function
$\Phi$ satisfies
$$
\left\{\begin{array}{l}
	d \int_{-\infty}^{0} J\left(x-y\right) \Phi(y){\rm d }y-d~ \Phi+c_{0} \Phi^{\prime}(x)-\nu 
	\Phi^{\prime}(x)+f(\Phi)=0, 
	~ -\infty<x<0, \\
	\Phi(0)=0.
\end{array}\right.
$$
And $0 \leq \Phi(x)<\Phi_{n}(x)$ for $x<0$. We extend 
$\Phi(x)$ for $x>0$ by $0$.

For fixed $\delta \in(0,1)$, let ${\Phi}_0(x)=\delta \Phi(x)$. By (\textbf{f3}), 
we 
obtain that $f(\delta \Phi) \geq \delta f(\Phi)$ and 
$$
\begin{cases}
	d\left(J * {\Phi}_0\right)(x)-d~ {\Phi }_0(x)+c_{0} 
	{\Phi_0}^{\prime}(x)-\nu {\Phi_0}^{\prime}(x)+f\left({\Phi_0}(x)\right) 
	\geq 0, &  -\infty<x<0, \\
	{\Phi}_0(x)=0, & x \geq 0.
\end{cases}
$$
Let $\Phi_*$ denote the traveling wave solution with minimal speed $c_0$ given by 
Proposition $\ref{prop1-3}$. For any $\sigma>0$, it follows
$$
\Phi_{*}(x-\sigma) \geq \Phi_{*}(-\sigma), \text { for } x \leq 0.
$$
Denote $\theta_{\sigma}(x):=\Phi_{*}(x-\sigma)-{\Phi}_0 \left(x\right)$. 
Since $\Phi_{*}(-\infty)=1$ and ${\Phi}_0 (x) \leq \delta <1$, for 
all large $\sigma>0$, we can obtain that
$
\theta_{\sigma}(x)\geq 0, \text { for } x \leq 
0.
$
Denote
\begin{equation}\label{eq1-8-8}
\sigma_{*}:=\inf \left\{\xi \in \mathbb{R}: \theta_{\sigma}(x) \geq 0 \text { for } 
x 
\leq 0 \text { and all } \sigma \geq \xi\right\}.
\end{equation}

If $\sigma_{*}=-\infty$, then ${\Phi}_0(x) \leq \Phi_{*}(x-\sigma)$ for all 
$\sigma \in \mathbb{R}$. Since
$\Phi_{*}(+\infty)=0$, then ${\Phi}_0(x) \leq 0$ as $\sigma \rightarrow-\infty$, which 
implies $\Phi(x) \equiv 0$.
If $\sigma_{*}>-\infty$, then
$
\theta_{\sigma_{*}}(x) \geq 0 \text { for } x \leq 0.
$
Since $\theta_{\sigma_{*}}(-\infty) \geq 1-\delta >0$ and 
$\theta_{\sigma_{*}}(0)=\Phi_{*}\left(-\sigma_{*}\right)>0$   
with $(\ref{eq1-8-8})$,  there is $x_{*} \in(-\infty, 0)$ such that
$
\theta_{\sigma_{*}}\left(x_{*}\right)=0.
$

Considering
$$
d \int_{-\infty}^{+\infty} J\left(x-y\right) \Phi_{*}\left(y-\sigma_{*}\right){ \rm d }y-d 
\Phi_{*}\left(x-\sigma_{*}\right)+(c_{0}-\nu)
\left(\Phi_{*}\right)^{\prime}\left(x-\sigma_{*}\right)
+f\left(\Phi_{*}\left(x-\sigma_{*}\right)\right)=0,~
  x \in \mathbb{R},
$$
we obtain
$$
\begin{cases}
	d\left(J * \Phi_{*}\right)\left(x-\sigma_{*}\right)-d 
\Phi_{*}\left(x-\sigma_{*}\right)+(c_{0}-\nu)\left(\Phi_{*}\right)^{\prime}\left(x-\sigma_{*}\right)+f\left(\Phi_{*}\left(x-\sigma_{*}\right)\right)=0,
 & -\infty<x<0,\\
 \Phi_{*}\left(x-\sigma_{*}\right)>0, & x \geq 0.
\end{cases}
$$
Applying the { Maximum principle} to $\theta_{\sigma_{*}}$, we can
conclude 
that 
$\theta_{\sigma_{*}}(x)>0$ for $x<0$, which yields a contradiction to 
$\theta_{\sigma_{*}}\left(x_{*}\right)=0$. Thus it always holds that $\Phi(x) \equiv 0$. Since $c_{n}$ is arbitrary and  
increasingly converges to $c_0$,  it follows that $(\ref{eq1-3-3})$ holds.

Step~2: We will prove that there is a 
unique $c_{**}$ such that 
$\left(c_{**}, \Phi^{c_{**}}(x)\right)$ satisfies the problem 
$(\ref{eq1-2})$ with $(\ref{eq1-202})$. 

For any $c \in\left(\nu, c_{0}\right)$, define
$$
\mathscr M(c):=\mu \int_{-\infty}^{0} \int_{0}^{\infty} J\left(x-y\right) \Phi^{c}(x){\rm d} y {\rm 
d} x.
$$
By Proposition~\ref{prop1-1}, we see that $\Phi^
{c}(x)$ is strictly decreasing in $c$, so does $\mathscr M(c)$. According to the uniqueness of $\Phi^{c}$, using  the  same 
arguments to show the convergence of $\Phi_{n}(x)$ in Step~1, we can  obtain
that $\Phi^{c}(x)$ is continuous in $c$ uniformly for $x$ in any compact set of 
$\left(-\infty, 0\right]$. It follows that $\mathscr M(c)$ is also continuous in $c$. 

Considering 
the function $$\Theta: c \mapsto c-\mathscr M(c), \text{ for }c \in\left(\nu, c_{0}\right),$$ one can see that $\Theta$ is 
continuous and strictly increasing in $c$. By $(\ref{eq1-3-3})$ and the Lebesgue dominated 
convergence 
theorem, it follows that $\Theta(c)\rightarrow c_{0}>0$ as $c \rightarrow {c_{0}}^{-}$. 
For fixed $\mathscr M\left(2\nu\right)>0$, 
set $$c_{min}=\min\{2\nu, \mathscr M(2\nu)\},$$
for any small $c\in\left(0,  c_{min} \right)$, according to the strict monotonicity of $\Phi^
{c}(x)$ in $c$,  it follows that $\Theta(c) \leq c-\mathscr M\left(2\nu\right)<0$. Then $\Theta(c)$ admits
 a unique root $c=c_{**} \in\left(0, c_{0}\right)$ satisfying
$c_{**}=\mathscr M\left(c_{**}\right)$. Thus, $(\ref{eq1-303})$ holds.
\end{proof}

Moreover, for the problem 
\begin{equation}\label{eq1-012}
	\left\{\begin{array}{l}d  \int_{-\infty}^{0} J\left(x-y\right) 
		\Phi(y) \mathrm{d} y-d~  \Phi+(c-\nu) \Phi^{\prime}(x)+f
		(\Phi(x) 
		)=0, 
		~-\infty<x<0, \\ \Phi(-\infty)=1, ~
		\Phi(0)={0},
\\
		c=\mu\int_{-\infty}^{0} \int_{0}^{\infty} 
		J\left(x-y\right)
		\Phi(x) \mathrm{d} y \mathrm{d} x,
  \end{array}\right.
\end{equation}
we have the following important result.
\begin{thm}
Assume that $(\mathbf{J}),(\mathbf{J_{*}})$ and $(\mathbf{f} 
	\mathbf{4})$ hold, for  
	$\tilde c_r:= \tilde c+\nu>0$,  the problem 
	$(\ref{eq1-012})$  admits a unique solution pair $\left(c_r^*,\Phi^{c_r^*}\right)$ with 
	$\Phi^{c_r^*}(x)$ 
	strictly decreasing and $c_r^*\in\left(0,\tilde c_r\right)$. 
\end{thm}
\begin{proof}
Since {($\mathbf{J_{**}}$)} implies {($\mathbf{J_{*}}$)}, according to  Lemma~$\ref{le-05.1}$, it suffices to explore the case in which {($\mathbf{J_{**}}$)} is not satisfied. The detailed proof is similar to the proof of Lemma~2.9 in \cite{du2021semi}. Here, the proof can be obtained only by making some minor modifications. 
\end{proof}
For the following problem
\begin{equation}\label{eq1-3}
	\left\{\begin{array}{l}d  \int_{-\infty}^{0} {J}\left(x-y\right)  
		\Phi\left(y\right) \mathrm{d} y-d~ \Phi+(c+\nu) \Phi^{\prime}(x)+f\left(\Phi(x) 
		\right)=0,
		~-\infty<x<0, \\ \Phi(-\infty)=1, ~ 
		\Phi(0)={0},\\
		c= \mu\int_{-\infty}^{0} \int_{0}^{\infty} 
		J\left(x-y\right) 
		\Phi\left(x\right) \mathrm{d} y \mathrm{d} x,
	\end{array}\right.
\end{equation}
we also have 
\begin{thm}
	Assume that {\rm($\textbf{J}$)}, {\rm ($\mathbf{J_*}$)} and $(\mathbf{f} 
	\mathbf{4})$ hold, for  
	$\tilde c_l:= \tilde c-\nu>0$,  the problem 
	$(\ref{eq1-3})$ admits a unique solution pair $\left(c_l^*,\Phi^{c_l^*}\right)$ with 
	$\Phi^{c_l^*}(x)$ 
	strictly decreasing and $c_l^*\in(0,\tilde c_l)$. 
\end{thm}

For any given $0<\epsilon\ll1,$ take 
$$
\begin{aligned}
&f_1(u)=f(u)+\dfrac{\epsilon}{1+\epsilon}u^2=u\left(1-\frac{1}{1+\epsilon}u\right),
\\&f_2(u)=f(u)-\dfrac{\epsilon}{1-\epsilon}u^2=u\left(1-\frac{1}{1-\epsilon}u\right), \\&\Phi_1=1+\epsilon,~\Phi_2=1-\epsilon,
\end{aligned}
$$
then $0$ and $\Phi_i$ are the zero solutions of $f_i$, 
for $i=1, 2.$    

Let  $\left({c}_{r,i}, 
\Phi_{r,i}(x)\right)$ be the solution pair of the problem $(\ref{eq1-012})$ with $f$ replaced by $f_i$ and $\Phi_{r,i}(-\infty)=\Phi_i,$ 
and $\left({c}_{l,i}, \Phi_{l,i}(x)\right)$ be the 
solution pair of the problem $(\ref{eq1-3})$ with $f$ replaced by ${f}_{i}$ and 
$\Phi_{l, i}(-\infty)=\Phi_i.$ 
Then $\Phi_{r, i}(x)$ and $\Phi_{l, i}(x)$ are strictly decreasing for $i=1,2.$ And
we obtain
\begin{prop}\label{prop1-2}
\begin{equation}\label{eq1-1-13}
{c}_{r,2}<c_{r}^{*}<{c}_{r,1}, ~ \lim _{\epsilon 
\rightarrow 0} {c}_{r,1}=
\lim _{\epsilon \rightarrow 0} 
{c}_{r,2}=c_{r}^{*},
\end{equation}
and
\begin{equation}\label{eq1-1-14}
{c}_{l,2}<c_{l}^{*}<{c}_{l,1}, ~ \lim _{\epsilon 
	\rightarrow 0} {c}_{l,1}=
\lim _{\epsilon \rightarrow 0} 
{c}_{l,2}=c_{l}^{*}.
\end{equation}
\end{prop}
\begin{proof}
First, we prove ${c}_{r,2}<c_{r}^{*}.$	 Take $\left({c}_{r,2}, 
\Phi_{r,2}\right)$ into $(\ref{eq1-012})$, we have 
\[d  \int_{-\infty}^{0} {J}\left(x-y\right)  
\Phi_{r,2}(y) \mathrm{d} y-d ~ \Phi_{r,2}+({c}_{r,2}-\nu) 
{\Phi_{r,2}^{\prime}}(x)+f\left(\Phi_{r,2}(x) 
\right)>0, \text{ for } x\in (-\infty,0).\]
Since $\Phi_{r,2}$ is strictly decreasing, it implies ${c}_{r,2}<c_{r}^{*}.$

Using the similar techniques to take $\left({c}_{r,1}, 
\Phi_{r,1}\right)$ into $(\ref{eq1-012})$, we get
\[d  \int_{-\infty}^{0} {J}\left(x-y\right)  
\Phi_{r,1}(y) \mathrm{d} y-d ~ \Phi_{r,1}+({c}_{r,1}-\nu) 
{\Phi_{r,1}^{\prime}}(x)+f(\Phi_{r,1}(x) 
)<0, \text{ for } x\in (-\infty,0).\] According to the strict monotonicity of $\Phi_{r,1}$ in $x$, it implies ${c}_{r,1}>c_{r}^{*}.$
Meanwhile, ${c}_{r,1}$ is continuous and increasing in $\epsilon$ and 
${c}_{r,2}$ is continuous and decreasing in $\epsilon$.
It follows that    $$\lim\limits _{\epsilon 
	\rightarrow 0} {c}_{r,1}=
\lim\limits_{\epsilon \rightarrow 0} 
{c}_{r,2}=c_{r}^{*}.$$   

Similarly, $(\ref{eq1-1-14})$ is satisfied.
\end{proof}
To complete the proof of the Theorem~\ref{th1-1}, we first show the following lemmas.
\begin{lem}\label{le5-1}
	\[\liminf\limits_{t\rightarrow\infty}\dfrac{h(t)}{t}\geq {c}_{r,2}.\]
\end{lem}
	\begin{proof}
		For any given $\epsilon>0,$ define 
		$$\underline U\left(t,x\right):=\Phi_{r,2}\left(x-{c}_{r,2}t\right), ~\underline h\left(t\right)=c_{r,2} t, \text{ for }
t>0, ~x\in\left(-\infty,{c}_{r,2}t\right],$$ where $\left({c}_{r,2},\Phi_{r,2} \right)$ is the solution pair of the following problem
	\begin{equation}\label{eq1-6}
		\left\{\begin{array}{l}d  \int_{-\infty}^{0} {J}\left(x-y\right)  
			\Phi\left(y\right) \mathrm{d} y-d~  \Phi\left(x\right)+\left(c-\nu\right) 
			{\Phi}^{\prime}\left(x\right)+f_2
			\left(\Phi\left(x\right) 
			\right)=0, ~
			-\infty<x<0, \\ \Phi\left(-\infty\right)=1-\epsilon, ~
			\Phi\left(0\right)={0},\\
			c=\mu\int_{-\infty}^{0} \int_{0}^{\infty} 
			J\left(x-y\right)\Phi(x){\rm d}y{\rm d}x .
		\end{array}\right.
	\end{equation}
Then we can get that $\underline 
U(t,x)\leq 1-\epsilon,\text{ for } t>0,~x\in(-\infty,{c}_{r,2}t].$ And
 \begin{equation}
	\begin{aligned}
	\underline U_t=&-{c}_{r,2}{\Phi_{r,2}^{\prime}}(x-{c}_{r,2}t)
	\\=& \int_{-\infty}^{0} {J}\left(x-{c}_{r,2}t-y\right)  
	\Phi_{r,2}(y) \mathrm{d} y-d~  \Phi_{r,2}\left(x-{c}_{r,2}t\right)-\nu 
	{\Phi_{r,2}^{\prime}}(x-{c}_{r,2}t)+f_2
	\left(\Phi_{r,2}\left(x-{c}_{r,2}t\right)\right)
	\\=& \int_{-\infty}^{{c}_{r,2}t} {J}\left(x-y\right)  
	\Phi_{r,2}\left(y-{c}_{r,2}t\right) \mathrm{d} y-d~  \Phi_{r,2}(x-{c}_{r,2}t)-\nu 
	{\Phi_{r,2}^{\prime}}(x-{c}_{r,2}t)+f_2
	\left(\Phi_{r,2}(x-{c}_{r,2}t)\right)
	\\\leq&\int_{-\infty}^{{c}_{r,2}t} {J}\left(x-y\right)  
	\underline U\left(t,y\right) \mathrm{d} y-d  ~\underline U(t,x)-\nu 
	\underline U_x\left(t, x\right )+f
	\left(\underline U\left(t, x\right)\right).
	\end{aligned}
\end{equation}
By Theorem~\ref{ths5-6.2}, when the spreading happens, we have $\lim\limits_{t\rightarrow\infty}u(t,x)=1    $
uniformly in any compact subset of $\mathbb R.$	 Thus, there is $T>0$ such that 
$$u(t,x)>1-\epsilon/2>\underline U(T,x),~t\geq T.$$

Moreover,
\begin{equation}
	\begin{aligned}
		c_{r,2}=&\mu\int_{-\infty}^{0} \int_{0}^{\infty} 
		J\left(x-y\right)
		\Phi_{r,2}(x) \mathrm{d} y \mathrm{d} x
		\\=& \mu\int_{-\infty}^{{c}_{r,2}t} \int_{{c}_{r,2}t}^{\infty} 
		J\left(x-y\right)
		\Phi_{r,2}(x-{c}_{r,2}t) \mathrm{d} y \mathrm{d} x
		\\=&\mu\int_{-\infty}^{{c}_{r,2}t} \int_{{c}_{r,2}t}^{\infty} 
		J\left(x-y\right)
		\underline U(t,x) \mathrm{d} y \mathrm{d} x.
	\end{aligned}
\end{equation}
Then by comparison principle, it follows 
\[\underline U(t,x)\leq u\left(t+T,x\right),~c_{r,2}t\leq h(t+T),~\text{for}~t>0,~x\in 
(-\infty,c_{r,2}t].\]
Thus, \begin{equation}\label{eq1-9}
	\liminf\limits_{t\rightarrow\infty}\dfrac{h(t)}{t}\geq c_{r,2}.
\end{equation}
	\end{proof}

\begin{lem}\label{le5-2}
	\[\limsup\limits_{t\rightarrow\infty}\dfrac{h(t)}{t}\leq {c}_{r,1}.\]
\end{lem}
\begin{proof}
	For the following problem
 $$
\begin{cases}
\tilde u^{\prime}(t)=f(\tilde u), & t>0,\\
\tilde u(0)=\|u_0\|_{\infty},
\end{cases}
$$
we can get $ u
\leq \tilde u$, which implies for any $\epsilon>0,$
 there is $\tilde T>0$ such that $$u(t,x)\leq 1+\epsilon/2,~\text{for}~t\geq\tilde 
 T,~x\in \left(-\infty,h(t)\right].$$

In view that $\left({c}_{r,1}, \Phi_{r,1}(x)\right)$ is a 
solution of problem $(\ref{eq1-012})$ with $f$ replaced by ${f}_{1}$ and 
$\Phi_{r,1}(-\infty)=1+\epsilon$. Hence there exists 
$\tilde{x}>h(\tilde{T})$ large enough such that
$$
u(\tilde{T}, x) \leq 
1+\epsilon/2<\Phi_{r,1}\left(x-\tilde{x}\right),~\text { for } 
x \in(-\infty, h(\tilde{T})].
$$
Let
$$
\overline U(t, x):=\Phi_{r,1}\left(x-{c}_{r,1} 
t-\tilde{x}\right),~\overline h(t)=c_{r,1} t+\tilde x, ~ \text { for } t>0,~x \in\left(-\infty, {c}_{r,1} 
t+\tilde{x}\right],
$$
then we have 
\begin{equation}
	\begin{aligned}
		\overline  
		U_t=&-{c}_{r,1}{\Phi_{r,1}^{\prime}}\left(x-{c}_{r,1}t-\tilde{x}\right)
		\\=& \int_{-\infty}^{0} {J}\left(x-{c}_{r,1}t-\tilde{x}-y\right)  
		\Phi_{r,1}(y) \mathrm{d} y-d~  \Phi_{r,1}(x-{c}_{r,1}t-\tilde{x})-\nu 
		{\Phi_{r,1}^{\prime}}(x-{c}_{r,1}t-\tilde{x})+f_1
		\left(\Phi_{r,1}(x-{c}_{r,1}t)\right)
		\\=& \int_{-\infty}^{{c}_{r,1}t+\tilde{x}} {J}\left(x-y\right)  
		\Phi_{r,1}\left(y-{c}_{r,1}t-\tilde{x}\right) \mathrm{d} y-d  ~
		\Phi_{r,1}(x-{c}_{r,1}t-\tilde{x})-\nu 
		{\Phi_{r,1}^{\prime}}\left (x-{c}_{r,1}t-\tilde{x}\right )+f_1
		\left(\Phi_{r,1}\left(x-{c}_{r,1}t-\tilde{x}\right)\right)
		\\\geq&\int_{-\infty}^{{c}_{r,1}t+\tilde{x}} {J}\left(x-y\right)  
		\overline  U\left(t,y\right) \mathrm{d} y-d ~ \overline  U(t,x)-\nu 
		\overline  U_x(t,x)+f
		(\overline  U(t,x)).
	\end{aligned}
\end{equation}
Moreover,
\begin{equation}
	\begin{aligned}
	\overline h^{\prime}(t)=c_{r,1}=&\mu\int_{-\infty}^{0} \int_{0}^{\infty} 
J\left(x-y\right) 
\Phi_{r,1}\left(x\right) \mathrm{d} y \mathrm{d} x
\\=& \mu\int_{-\infty}^{{c}_{r,1}t+\tilde{x}} 
\int_{{c}_{r,1}t+\tilde{x}}^{\infty} 
J\left(x-y\right) 
\Phi_{r,1}\left(x-{c}_{r,1}t-\tilde{x}\right) \mathrm{d} y \mathrm{d} x
\\=&\mu\int_{-\infty}^{{c}_{r,1}t+\tilde{x}} 
\int_{{c}_{r,1}t+\tilde{x}}^{\infty} 
J\left(x-y\right) 
\overline U\left(t,x\right) \mathrm{d} y \mathrm{d} x.
	\end{aligned}
\end{equation}
Thus,
applying the 
comparison principle, we have
$$ u(t+\tilde{T}, x) \leq \overline U\left(t, x\right),~h(t+\tilde{T}) \leq {c}_{r,1} t+\tilde{x},$$ for  $t>0$ and $x \in(-\infty, 
h(t+\tilde{T})]$.
It yields 
\begin{equation}\label{eq1-13}
\limsup _{t \rightarrow \infty} \frac{h(t)}{t} \leq {c}_{r,1}.
\end{equation}
\end{proof}
\begin{thm}\label{rem1-3}
Since $\epsilon>0$ is arbitrarily chosen and small enough, combining $(\ref{eq1-9})$ 
and 
$(\ref{eq1-13})$,  it follows
$$
\lim _{t \rightarrow \infty} \frac{h(t)}{t}=c_{r}^{*}.
$$
\end{thm}	

\begin{lem}\label{le5-3}
\[\liminf\limits_{t\rightarrow\infty}\dfrac{-g(t)}{t}\geq c_{l,2}.\]	
\end{lem}
\begin{proof}
	Let    $$\underline V(t,x)=\Phi_{l,2}(-x-c_{l,2}t),~\underline g(t)=-c_{l,2}t,~\text{ for } t>0, ~x\in[-c_{l,2}t, 
	\infty),$$   
then $\underline V(t,x)\leq 1-\epsilon.$ And by $(\ref{eq1-3})$, explicit calculations give
	\begin{equation}
		\begin{aligned}
			\underline V_t=&-{c}_{l,2}{\Phi_{l,2}}^{\prime}\left(-x-{c}_{l,2}t\right)
\\=& \int_{-\infty}^{0} {J}\left(-x-{c}_{l,2}t-y\right)  
\Phi_{l,2}\left(y\right) \mathrm{d} y-d  \Phi_{l,2}\left(-x-{c}_{l,2}t\right)+\nu 
{\Phi_{l,2}^{\prime}}\left(-x-{c}_{l,2}t\right)+f_2
\left(\Phi_{l,2}\left(-x-{c}_{l,2}t\right)\right)
\\=& \int_{-{c}_{l,2}t}^{\infty} {J}\left(x-y\right)  
\Phi_{l,2}\left(-y-{c}_{l,2}t\right) \mathrm{d} y-d  
\Phi_{l,2}\left(x-{c}_{l,2}t\right)+\nu 
{\Phi_{l,2}^{\prime}}\left(-x-{c}_{l,2}t\right)+f_2
\left(\Phi_{l,2}\left(x-{c}_{l,2}t\right)\right)
\\\leq&\int_{-{c}_{l,2}t}^{\infty} {J}\left(x-y\right)  
\underline V\left(t,y\right) \mathrm{d} y-d  \underline V\left(t,x\right)-\nu 
\underline V_x\left(t,x\right)+f
\left(\underline V\left(t,x\right)\right).		\end{aligned}
	\end{equation}
Since $\lim\limits_{t\rightarrow\infty}u(t,x)=1$
uniformly in any compact subset of $\mathbb R$ as the spreading occurs,  there is $ \bar T>0$ such 
that 
$$u(t,x)>1-\epsilon/2>\underline V\left(\bar T,x\right),~t\geq \bar T.$$ 

Moreover,
\begin{equation}
	\begin{aligned}
	-\underline g^{\prime}(t)=c_{l,2}=&\mu\int_{-\infty}^{0} \int_{0}^{\infty} 
J\left(x-y\right) 
\Phi_{l,2}\left(x\right) \mathrm{d} y \mathrm{d} x
\\=& \mu\int_{\infty}^{{-c}_{l,2}t} \int_{-{c}_{l,2}t}^{-\infty} 
J\left(x-y\right) 
\Phi_{r,2}\left(-x-{c}_{r,2}t\right) \mathrm{d} y \mathrm{d} x
\\=&\mu\int_{-{c}_{l,2}t}^{\infty} \int_{-\infty}^{-{c}_{l,2}t} 
J\left(x-y\right) 
\underline V\left(t,x\right) \mathrm{d} y \mathrm{d} x.
	\end{aligned}
\end{equation}
Then by comparison principle, it follows 
\[\underline V(t, x)\leq u(t+\bar T,x),~-c_{l,2}t\geq g(t+\bar T),~\text{for}~t>0,~x\in 
[-c_{l,2}t,\infty).\]
Thus, \begin{equation}\label{eq1-16}
	\liminf\limits_{t\rightarrow\infty}\dfrac{-g(t)}{t}\geq c_{l,2}.
\end{equation}
\end{proof}

\begin{lem}\label{le5-4}
\[\limsup\limits_{t\rightarrow\infty}\dfrac{-g(t)}{t}\leq c_{l,1}.\]	
\end{lem}
\begin{proof}
	For the following problem
	\begin{equation}
	\begin{cases}
		\hat u^{\prime}(t)=f(\hat u), & t>0,
		\\ ~\hat 
		u(0)=\|u_0\|_{\infty},
	\end{cases}
	\end{equation}
	we can get $ u
	\leq \hat u$, which implies for any $\epsilon>0,$
	there is $\hat T>0$ such that $$u(t,x)\leq 1+\epsilon/2,~\text{for}~t\geq\hat 
	T,~x\in [g(t),\infty).$$
	
	In view that $\left({c}_{l,1}, \Phi_{l,1}(x)\right)$ is a 
	solution of problem $(\ref{eq1-3})$ with $f$ replaced by ${f}_{1}$ and 
	$\Phi_{l,1}(-\infty)=1+\epsilon$. Hence there exists 
	$\hat{x}>-g(\hat{T})$ large such that
	$$
	u(\hat{T}, x) \leq 
	1+\epsilon/2<\Phi_{l,1}(-x-\hat{x}),~\text { for } 
	x \in[g(\hat{T}),\infty).
	$$
	Let
	$$
	\overline V(t, x):=\Phi_{l,1}\left(-x-{c}_{l,1} 
	t-\hat{x}\right), ~\overline g(t)=-{c}_{l,1} 
	t-\hat{x}, ~ \text { for } t>0, ~x \in\left[-{c}_{l,1} 
	t-\hat{x},\infty\right),
	$$
	then we have 
	\begin{equation}
		\begin{aligned}
			\overline  
			V_t=&-{c}_{l,1}{\Phi_{l,1}^{\prime}}\left(-x-{c}_{l,1}t-\hat{x}\right)
			\\=& \int_{-\infty}^{0} {J}\left(-x-{c}_{l,1}t-\hat{x}-y\right)  
			\Phi_{l,1}(y) \mathrm{d} y-d~  
			\Phi_{l,1}\left(-x-{c}_{l,1}t-\hat{x}\right)+\nu 
			{\Phi_{l,1}^{\prime}}\left(-x-{c}_{l,1}t-\hat x\right)+f_1
			\left(\Phi_{l,1}(-x-{c}_{l,1}t-\hat x)\right)
			\\=& \int_{-{c}_{l,1}t-\hat{x}}^{\infty} {J}\left(x-y\right)  
			\Phi_{l,1}\left(-y-{c}_{l,1}t-\hat{x}\right) \mathrm{d} y-d ~ 
			\Phi_{l,1}\left(-x-{c}_{l,1}t-\hat{x}\right)+\nu 
			{\Phi_{l,1}^{\prime}}\left(-x-{c}_{l,1}t-\hat{x}\right)+f_1
			(\Phi_{l,1}\left(-x-{c}_{l,1}t-\hat{x})\right)
			\\\geq&\int_{-{c}_{l,1}t-\hat{x}}^{\infty} {J}\left(x-y\right)  
			\overline  V\left(t,y\right) \mathrm{d} y-d ~ \overline  V(t,x)-\nu 
			\overline  V_x\left(t,x\right)+f
			\left(\overline  V(t,x)\right).
		\end{aligned}
	\end{equation}	
	Moreover,
	\begin{equation}
		\begin{aligned}
		-\overline g^{\prime}(t)=c_{l,1}=&\mu\int_{-\infty}^{0} \int_{0}^{\infty} 
			J\left(x-y\right) 
			\Phi_{l,1}(x) \mathrm{d} y \mathrm{d} x
			\\=& \mu\int_{\infty}^{-{c}_{l,1}t-\hat x} 
			\int_{-{c}_{l,1}t-\hat x}^{-\infty} 
			J\left(x-y\right)
			\Phi_{l,1}(-x-{c}_{r,1}t-\hat {x}) \mathrm{d} y \mathrm{d} x
			\\=&\mu\int_{-{c}_{l,1}t-\hat x }^{\infty} 
			\int_{-\infty}^{-{c}_{l,1}t-\hat x} 
			J\left(x-y\right)
			\overline V(t,x) \mathrm{d} y \mathrm{d} x.
		\end{aligned}
	\end{equation}
	Thus 
	applying the 
	comparison principle, we have
	$$ 
	u(t+\hat{T}, x) \leq \overline V(t, x),~g(t+\hat{T}) \geq -{c}_{l,1} t-\hat{x},$$ for $t>0$ and $x 
	\in[g(t+\hat{T}),\infty)$.
	It yields 
	\begin{equation}\label{eq1-20}
		\limsup_{t \rightarrow \infty} \frac{-g(t)}{t} \leq {c}_{l,1}.
	\end{equation}
\end{proof}
\begin{thm}\label{rem1-4}
	Since $\epsilon>0$ is arbitrarily  chosen and small enough, combining 
	$(\ref{eq1-16})$ 
	with 
	$(\ref{eq1-20})$,  we have
	$$
	\lim_{t \rightarrow \infty} \frac{-g(t)}{t}=c_{l}^{*}.
	$$
\end{thm}
According to the above several results, now it turns to complete the proof of Theorem~$\ref{th1-1}$.
\begin{proof}[Proof of Theorem~$\ref{th1-1}$]
	According to Proposition~$\ref{prop1-2}$,
	Theorems~$\ref{rem1-3}$ and $~\ref{rem1-4}$, 
	it suffices to prove
	\begin{align*}
		0<c_{l}^*<c^*<c_{r}^*.
			\end{align*}
For the problem
\begin{equation}
	\left\{\begin{array}{l}\label{eq1-1-32}
		d  \int_{-\infty}^{0} {J}\left(x-y\right)  
		\Phi(y) \mathrm{d} y-d~  \Phi+c~ \Phi^{\prime}(x)+f
		(\Phi(x) 
		)=0, 
		\text { 
		}-\infty<x<0, \\ \Phi(-\infty)=1, ~
		\Phi(0)={0},
	\end{array}\right.
\end{equation}	
as is stated in Theorem~2.7~\cite{du2021semi},  there is a $\tilde c>0$ such that
problem $(\ref{eq1-1-32})$ admits a 
solution 
$(c,\Phi^c(x))$ for 
any $c\in (0,\tilde c)$ with $\Phi^c(x)$ is nonincreasing in $c$ for fixed 
$x\in(-\infty,0].$
And \begin{align}\label{eq2-19}
	\lim\limits_{c\rightarrow\tilde c^{-}}\Phi^c(x)=0~\text{locally uniformly 
	in}~ 
	x 
	\in (-\infty,0].
\end{align}
Further, for any $\mu>0$,
  there exists a unique $c^*=c^*(\mu)\in\left(0, 
\tilde c\right)$  such that
\begin{equation}\label{eq5-032}
c^*=\mu \int_{-\infty}^{0} \int_{0}^{\infty} J\left(x-y\right) \Phi^{c^*}(x) {\rm d} y{ 
\rm d} x.
\end{equation}
Actually, suppose that {\rm($\textbf{J}$)}, {($\mathbf{J_*}$)} and $(\mathbf{f} 
	\mathbf{4})$ hold, $c^*$ is the asymptotic spreading speed of the rightward front and the leftward front for $(\ref{eq1-1})$ without the advection term ($\nu=0$).

According to Lemma~\ref{le-05.1}, for any  $c\in(0,\tilde c_r)$, let $(c,\Phi_r^c)$ be the solution of $(\ref{eq1-2})$. Similarly, $(c,\Phi^c)$ and  $(c,\Phi_l^c)$ 
satisfy the corresponding equations, where $c\in(0,\tilde c)$ and $c\in(0,\tilde c_l)$, respectively.

Denote 
$$
\begin{aligned}
\mathscr{F}_l(c):= \int_{-\infty}^{0} \int_{0}^{\infty} J\left(x-y\right) 
\Phi_l^{c}(x) {\rm d} y{ 
	\rm d}x, & \text{ for } c\in(0, \tilde c_l), \\
\mathscr{F}(c):= \int_{-\infty}^{0} \int_{0}^{\infty} 
	J\left(x-y\right) 
\Phi^{c}(x) {\rm d} y{ 
\rm d}x, & \text{ for } c\in(0, \tilde c),\\
\mathscr{F}_r(c):= \int_{-\infty}^{0} \int_{0}^{\infty} 
J\left(x-y\right) 
\Phi_r^{c}(x) {\rm d} y{ 
\rm d}x, &\text{ for } c \in (0,\tilde c_r).
\end{aligned}
$$

Given $(\ref{eq1-3-3})$ and $(\ref{eq2-19})$, applying the Lebesgue dominated 
convergence 
theorem, we have the following facts:

~~~~(1) 	$\lim\limits_{c\rightarrow 
	{\tilde c_l}^-}\mathscr {F}_l\left(c\right)=0,~\lim\limits_{c\rightarrow {\tilde c}^-}\mathscr 
	{F}(c)=0,~ 
\lim\limits_{c\rightarrow {\tilde c_r}^-}\mathscr {F}_r(c)=0;$\vspace{0.1cm}

~~~~(2) 	$\mathscr {F}_l(c-\nu)=\mathscr 
	{F}(c)=\mathscr {F}_r(c+\nu), \text{ for } c\in(\nu,\tilde c);$\vspace{0.1cm}

~~~~(3)    $\mathscr F_l(c)$ (resp. $\mathscr F(c)$, $\mathscr 
F_r(c)$) 
is  continuous and strictly decreasing in  
$c\in(0,\tilde c_l)$ (resp. $c\in(0, \tilde c),~c\in (0,\tilde c_r)$);\vspace{0.1cm}

~~~~(4)  $\mathscr {F}_l(c)<\mathscr {F}(c)$ for $c\in (0,\tilde c_l)$ and $\mathscr {F}(c)<\mathscr 
{F}_r(c)$ for $c\in[\nu, \tilde c)$.
 \newline Therefore, $\mathscr{F}_l(c)$ (resp. $\mathscr{F}(c)$ and 
 $\mathscr{F}_r(c)$   ) contact the line  $\mathscr {F}=\frac{c}{\mu}$ at the point 
 $\left(c_l^*, \mathscr{F}_l \left(c_l^*\right)\right)$ (resp. $\left(c^*, \mathscr{F}\left(c^*\right)\right)$ and $\left(c_r^*, 
 \mathscr{F}_r\left(c_r^*\right)\right)$   ) with $c_l^*<c^*<c_r^*$. 
 
For fixed $\mu$, according to the above analysis, it can be easily seen that $c_r^*$ is strictly increasing in $\nu$ and $c_l^*$ is strictly decreasing in $\nu$, and 
$$\lim\limits_{\nu\rightarrow 0}c_r^*=\lim\limits_{\nu\rightarrow 0}c_l^*=c^*.$$
Thus, the 
 Theorem~$\ref{th1-1}$ is proved.
\end{proof}
\begin{rem}
According to Theorem~$\ref{th1-1}$, the double free boundaries of the problem 
$(\ref{eq1-1})$ move at different finite speeds as $t\rightarrow \infty$ under the effect of the advection, comparing with the non-advection case that the leftward front and the rightward front spread at the same speed $c^*$.
\end{rem}
Actually,  let $(u,g,h)$ be the solution of problem $(\ref{eq1-1})$, when the spreading 
happens, we can further show 
\begin{prop}
\[c_l^*\geq c^*-\nu \text{ and } c_r^*\leq c^*+\nu.\]
\end{prop}
\begin{proof}
Set \[\overline h(t)=\left(c^*+\nu\right)t+L,~
\overline U(t,x)=(1+\epsilon)\Phi^{c^*}\left(x-\overline h(t)\right),\]
where  $\left(c^*, \Phi^{c^*}(x)\right)$ denotes the solution pair of $(\ref{eq1-1-32})$ with $(\ref{eq5-032})$, and $L$, $\epsilon$ will be chosen later.

In view of $(\ref{eq5-12})$, $\limsup\limits_{t\rightarrow \infty} U(t,x)\leq 
1,$ there exists 
$T\gg1$ such that $$U(t+T,x)\leq 
1+\epsilon/3,~\text{for}~t\geq0,~x\in 
\left[g\left(t+T\right),h\left(t+T\right)\right].$$

Take $L$ large enough such that $\overline h(0)=L>h(T),$  and
$$\overline U(0,x)=(1+\epsilon)\Phi^{c^*}\left(x-\overline 
h(0)\right)=(1+\epsilon)\Phi^{c^*}\left(x-L\right)>\left(1+\epsilon\right)\left(1-\epsilon/3\right)>U(T,x),~ 
x\in\left[g\left(T\right),h\left(T\right)\right].$$

 Next, we prove $$\overline U_t\geq 
d\int_{g(t+T)}^{\overline 
h(t)}J\left(x-y\right)\overline U\left(t,y\right){\rm d}y-d ~\overline U(t,x)-\nu \overline 
U_x+f\left(\overline U\left(t,x\right)\right),$$
for $t>0$ and $x\in\left[g\left(t+T\right),\overline h\left(t\right)\right].$

In view of $(\textbf{f4})$, it follows $$f\left(\overline U(t,x)\right)=f\left((1+\epsilon)\Phi^{c^*}\left(x-\overline 
h(t)\right)\right)\leq (1+\epsilon)f\left(\Phi^{c^*}\left(x-\overline h(t)\right)\right).$$  
Since $\Phi(x)$ is nonincreasing,  
then direct calculations give
\begin{equation}
\begin{aligned}
	\overline U_t=&-(1+\epsilon){\Phi^{c^*}}^{\prime}\left(x-\overline h(t)\right)\left(c^*+\nu\right)\\
	=&(1+\epsilon)\left[d\int_{-
	\infty}^{\overline h(t)}J\left(x-y\right)\Phi^{c^*}\left(y-\overline h(t)\right){\rm d}y-d~ \Phi^{c^*}(x-\overline h(t))+f\left(\Phi^{c^*}(x-\overline h(t))\right)\right]\\
	-&(1+\epsilon)\nu{\Phi^{c^*}}^{\prime}\left(x-\overline h(t)\right)\\
	\geq &d\int_{-\infty}^{\overline h(t)}J\left(x-y\right)\overline U\left(t,y\right){\rm d}y-d~ \overline U(t,x)-\nu\overline U_x(t,x)+f\left(\overline U(t,x)\right)\\
	\geq&d\int_{g(t+T)}^{\overline h(t)}J\left(x-y\right)\overline U(t,y){\rm 
	d}y-d~\overline U\left(t,x\right)-\nu \overline U_x(t,x)+f\left(\overline U(t,x)\right).
\end{aligned}
\end{equation}

And if we take $\epsilon<\nu/c^*,$ then
\begin{equation}
    \begin{aligned}
    \overline h^{\prime}(t)=&c^*+\nu>c^*(1+\epsilon)\\=&
    \mu(1+\epsilon) \int_{-\infty}^{0}\int_{0}^{\infty}J\left(x-y\right)\Phi^{c^*}(x) {\rm d}y {\rm d}x
    \\\geq&\mu \int_{-\infty}^{\overline h(t)}\int_{\overline h(t)}^{\infty}J\left(x-y\right)\overline U(t,x){\rm d}y {\rm d}x
    \\\geq& \mu \int_{g(t+T)}^{\overline h(t)}\int_{\overline h(t)}^{\infty}J\left(x-y\right)\overline U(t,x){\rm d}y {\rm d}x.
    \end{aligned}
\end{equation}
Next, we aim to prove 
\begin{equation}
    \overline h(t)>h(t+T),~\overline U\left(t,x\right)>U\left(t+T,x\right),~{\rm for}~ t>0,~x\in\left[ g\left(t+T\right),h\left(t+T\right) \right].
\end{equation}
Suppose that the above inequalities do not hold for all $t>0,$ there exists a 
first time $t^*>0$ such that $ \overline h\left(t^*\right)=h\left(t^*+T\right),$ or $\overline 
h(t^*)>h(t+t^*)$ but $\overline U\left(t^*,x^*\right)=U\left(t^*+T,x^*\right),$ for some $x^*\in 
\left[g\left(t^*+T\right),h\left(t^*+T\right)\right].$

If $\overline h\left(t^*\right)=h\left(t^*+T\right),$ then 
\begin{equation}\label{e6-7}
\overline h^{\prime}\left(t^*\right)\leq h^{\prime}\left(t^*+T\right).
\end{equation}
Considering that $\overline U\left(t^*,x\right)\geq 
U\left(t^*+T, x\right),~x\in\left[g\left(t^*+T\right),h\left(t^*+T\right)\right],$ we can get that
\begin{equation}
\begin{aligned}
\overline{h}^{\prime}\left(t^{*}\right) &>\mu 
\int_{g\left(t^{*}+T\right)}^{\overline{h}\left(t^{*}\right)} \int_{\overline 
h\left(t^{*}\right)}^{\infty} J\left(x-y\right) \overline{U}\left(t^{*}, x\right){ \rm d} 
y {\rm d} x 
\\
&=\mu \int_{g\left(t^{*}+T\right)}^{h(t^*+T)} 
\int_{h\left(t^{*}+T\right)}^{\infty} J\left(x-y\right) \overline{U}\left(t^{*}, x\right) 
{ \rm d} y {\rm d} x  \\
& \geq \mu\int_{g(t^*+T)}^{h(t^*+T)} \int_{h\left(t^{*}+T\right)}^{\infty} 
J\left(x-y\right) U\left(t^{*}+T, x\right) { \rm d} y {\rm d} x  \\
&=h^{\prime}\left(t^{*}+T\right),
\end{aligned}
\end{equation}
which yields a contradiction to $(\ref{e6-7})$.

If $\overline h(t^*)>h(t+t^*)$ and 
\begin{equation}\label{e6-9}
\overline U\left(t^*,x^*\right )=U\left(t^*+T,x^*\right),
\end{equation}
since $\overline U(t,x)>0,~x=g\left(t+T\right)$ or $x=h\left(t+T\right),~t\in \left(0,t^*\right]$, and $\overline U\left(0,x\right)>U\left(T, x\right),~x\in \left[g\left(T\right),h\left(T\right)\right]$,  by comparison principle, we can obtain that 
\begin{equation}
    \overline U\left(t^*,x\right)>U\left(t^*+T, x\right),~x\in\left[g\left(t+T\right),h\left(t+T\right)\right],
\end{equation}
which contradicts $(\ref{e6-9})$.  
Thus, $$c_r^*=\limsup\limits_{t\rightarrow 
\infty}\dfrac{h(t)}{t}\leq\lim\limits_{t\rightarrow\infty}\dfrac{\overline 
h(t-T)}{t}=c^*+\nu.$$
The case about $c_l^*\geq c^*-\nu$ can be easily proved by constructing the corresponding lower solution.
\end{proof}
\begin{proof}[The proof of Theorem~$\ref{th7.8}$]
	If $J(x)$ does not satisfy
	$\left(\mathbf{J_*}\right)$, according to the Theorem~$\ref{th1-1},$ based on the Section~3.3 in~\cite{DU2022369}, we can complete the proof of this theorem by some subtle constructions. Now we only provide several important sketches. 
	Firstly,  construct a series of cut-off functions $J_n(x)$ which satisfies the condition $\left(\mathbf{J_*}\right)$. 
choose a nonnegative, even function sequence $\{J_n\}$ such that  each $J_n(x) \in C^1$ has  nonempty
compact support, and
\begin{equation}
J_n(x)\leq J_{n+1}(x)\leq J(x), \text{ and }  J_n(x)\rightarrow J(x), \text { in } L^1(\mathbb R) \text{ as } n\rightarrow \infty.
\end{equation}
where $J_n(x)= J(x)\chi_n(x)$  and $\{\chi_n\}$  is a properly smooth cut-off function sequences
such that $J_n(x)$  satisfies $\left(\mathbf{J_*}\right)$.	
	Next,  rewrite the problem $(\ref{eq1-1})$ with $J(x)$ replaced by $J_n(x)$, and we can get the spreading speed $c_n$ of the corresponding semi-wave problem.   Similar to prove Lemmas~$\ref{le5-1}$ and $\ref{le5-3}$, we can show that $$\liminf\limits_{t\rightarrow \infty} \frac{h_n(t)}{t}\geq c_{r,2}^n, \text{ and } \liminf\limits_{t\rightarrow \infty} \frac{g_n(t)}{t}\geq c_{l,2}^n.$$
 Where $c_{r,2}^n$ and $c_{l,2}^n$ satisfy the equations $(\ref{eq1-012})$ and $(\ref{eq1-3})$ with $J$ replaced by $J_n$ and $f$ replaced by $f_2$, respectively.
 In view that $J(x)$ does not satisfy $\left(\mathbf{J_*}\right)$,  it can be proved that $\lim\limits_{n\rightarrow\infty} c_{r,2}^n=\infty$ and  $\lim\limits_{n\rightarrow\infty} c_{l,2}^n=\infty$ by contradiction.   Here, we omit the detailed steps. 
\end{proof}

\begin{rem}
According to Theorems~$\ref{th1-1}$ and $\ref{th7.8},$  as the spreading occurs, the assumption $\left(\mathbf{J_*}\right)$ is a threshold condition to determine whether the spreading speed is finite or not.
\end{rem}

\section*{Acknowledges} 
This work is supported by the China Postdoctoral Science Foundation (No. 2022M710426) and the Postdoctoral Science Foundation Project of Beijing Normal University at Zhuhai.

\bibliography{refernon}

\begin{thebibliography}{10}
\expandafter\ifx\csname url\endcsname\relax
  \def\url#1{\texttt{#1}}\fi
\expandafter\ifx\csname urlprefix\endcsname\relax\def\urlprefix{URL }\fi
\expandafter\ifx\csname href\endcsname\relax
  \def\href#1#2{#2} \def\path#1{#1}\fi

\bibitem{furter1989local}
J.~Furter, M.~Grinfeld, Local vs. non-local interactions in population
  dynamics, Journal of Mathematical Biology 27~(1) (1989) 65--80.

\bibitem{andreu2010nonlocal}
F.~Andreu-Vaillo, J.~M. Maz{\'o}n, J.~D. Rossi, J.~J. Toledo-Melero, Nonlocal
  diffusion problems, no. 165, American Mathematical Society, 2010.

\bibitem{kao2010random}
C.-Y. Kao, Y.~Lou, W.~Shen, Random dispersal vs. non-local dispersal, Discrete
  \& Continuous Dynamical Systems 26~(2) (2010) 551--596.

\bibitem{massaccesi2017nonlocal}
A.~Massaccesi, E.~Valdinoci, Is a nonlocal diffusion strategy convenient for
  biological populations in competition?, Journal of Mathematical Biology
  74~(1) (2017) 113--147.

\bibitem{berestycki2009non}
H.~Berestycki, G.~Nadin, B.~Perthame, L.~Ryzhik, The non-local {Fisher--KPP}
  equation: travelling waves and steady states, Nonlinearity 22~(12) (2009)
  2813--2844.

\bibitem{berestycki2016persistence}
H.~Berestycki, J.~Coville, H.-H. Vo, Persistence criteria for populations with
  non-local dispersion, Journal of Mathematical Biology 72~(7) (2016)
  1693--1745.

\bibitem{zhao2020dynamics}
M.~Zhao, Y.~Zhang, W.-T. Li, Y.~Du, The dynamics of a degenerate epidemic model
  with nonlocal diffusion and free boundaries, Journal of Differential
  Equations 269~(4) (2020) 3347--3386.

\bibitem{du2020analysis}
Y.~Du, W.~Ni, Analysis of a {West Nile} virus model with nonlocal diffusion and
  free boundaries, Nonlinearity 33~(9) (2020) 4407--4448.

\bibitem{li2020systems}
L.~Li, W.~Sheng, M.~Wang, Systems with nonlocal vs. local diffusions and free
  boundaries, Journal of Mathematical Analysis and Applications 483~(2) (2020)
  123646.

\bibitem{pu2021west}
L.~Pu, Z.~Lin, Y.~Lou, A {West Nile} virus nonlocal model with free boundaries
  and seasonal succession, arXiv preprint arXiv:2110.08055 (2021).

\bibitem{li2022free}
L.~Li, M.~Wang, Free boundary problems of a mutualist model with nonlocal
  diffusion, Journal of Dynamics and Differential Equations (2022) 1--29.

\bibitem{du2022two}
Y.~Du, M.~Wang, M.~Zhao, Two species nonlocal diffusion systems with free
  boundaries, Discrete \& Continuous Dynamical Systems 42~(3) (2022)
  1127--1162.

\bibitem{li2022dynamics}
L.~Li, W.-T. Li, M.~Wang, Dynamics for nonlocal diffusion problems with a free
  boundary, Journal of Differential Equations 330 (2022) 110--149.

\bibitem{du2022high}
Y.~Du, W.~Ni, The high dimensional {Fisher-KPP }nonlocal diffusion equation
  with free boundary and radial symmetry, part 1, SIAM Journal on Mathematical
  Analysis 54~(3) (2022) 3930--3973.

\bibitem{gu2014long}
H.~Gu, Z.~Lin, B.~Lou, Long time behavior of solutions of a
  diffusion--advection logistic model with free boundaries, Applied Mathematics
  Letters 37 (2014) 49--53.

\bibitem{ge2015sis}
J.~Ge, K.~I. Kim, Z.~Lin, H.~Zhu, A {SIS} reaction--diffusion--advection model
  in a low-risk and high-risk domain, Journal of Differential Equations
  259~(10) (2015) 5486--5509.

\bibitem{cui2016spatial}
R.~Cui, Y.~Lou, A spatial {SIS} model in advective heterogeneous environments,
  Journal of Differential Equations 261~(6) (2016) 3305--3343.

\bibitem{tian2018}
C.~Tian, S.~Ruan, On an advection--reaction--diffusion competition system with
  double free boundaries modeling invasion and competition of aedes albopictus
  and aedes aegypti mosquitoes, Journal of Differential Equations 265 (2018)
  4016--4051.

\bibitem{cheng2021dynamics}
C.~Cheng, Z.~Zheng, Dynamics and spreading speed of a reaction-diffusion system
  with advection modeling {West Nile} virus, Journal of Mathematical Analysis
  and Applications 493~(1) (2021) 124507.

\bibitem{lutscher2005effect}
F.~Lutscher, E.~Pachepsky, M.~A. Lewis, The effect of dispersal patterns on
  stream populations, SIAM Review 47~(4) (2005) 749--772.

\bibitem{lutscher2006effects}
F.~Lutscher, M.~A. Lewis, E.~McCauley, Effects of heterogeneity on spread and
  persistence in rivers, Bulletin of Mathematical Biology 68~(8) (2006)
  2129--2160.

\bibitem{JIANG2021103350}
D.~Jiang, K.-Y. Lam, Y.~Lou, Competitive exclusion in a nonlocal
  reaction--diffusion--advection model of phytoplankton populations, Nonlinear
  Analysis: Real World Applications 61 (2021) 103350.

\bibitem{yan2022competition}
X.~Yan, H.~Nie, P.~Zhou, On a competition-diffusion-advection system from river
  ecology: mathematical analysis and numerical study, SIAM Journal on Applied
  Dynamical Systems 21~(1) (2022) 438--469.

\bibitem{maidana2009spatial}
N.~A. Maidana, H.~M. Yang, Spatial spreading of {West Nile} virus described by
  traveling waves, Journal of Theoretical Biology 258~(3) (2009) 403--417.

\bibitem{fisher1937wave}
R.~A. Fisher, The wave of advance of advantageous genes, Annals of Eugenics
  7~(4) (1937) 355--369.

\bibitem{kolmogorov1937etude}
A.~N. Kolmogorov, {\'E}tude de l'{\'e}quation de la diffusion avec croissance
  de la quantit{\'e} de mati{\`e}re et son application {\`a} un probl{\`e}me
  biologique, Bull. Univ. Moskow, Ser. Internat., Sec. A 1 (1937) 1--25.

\bibitem{li2005spreading}
B.~Li, H.~F. Weinberger, M.~A. Lewis, Spreading speeds as slowest wave speeds
  for cooperative systems, Mathematical Biosciences 196~(1) (2005) 82--98.

\bibitem{weinberger2007anomalous}
H.~F. Weinberger, M.~A. Lewis, B.~Li, Anomalous spreading speeds of cooperative
  recursion systems, Journal of Mathematical Biology 55~(2) (2007) 207--222.

\bibitem{lewis2002spreading}
M.~A. Lewis, B.~Li, H.~F. Weinberger, Spreading speed and linear determinacy
  for two-species competition models, Journal of Mathematical Biology 45~(3)
  (2002) 219--233.

\bibitem{liang2022propagation}
X.~Liang, Z.~Tao, Propagation of {KPP} equations with advection in
  one-dimensional almost periodic media and its symmetry, Advances in
  Mathematics 407~(108568) (2022).

\bibitem{du2021semi}
Y.~Du, F.~Li, M.~Zhou, Semi-wave and spreading speed of the nonlocal
  {Fisher-KPP} equation with free boundaries, Journal de Math{\'e}matiques
  Pures et Appliqu{\'e}es 154 (2021) 30--66.

\bibitem{cao2019}
J.~F. Cao, Y.~Du, F.~Li, W.~T. Li, The dynamics of a {Fisher-KPP} nonlocal
  diffusion model with free boundaries, Journal of Functional Analysis 277~(8)
  (2019) 2772--2814.

\bibitem{du2010spreading}
Y.~Du, Z.~Lin, Spreading-vanishing dichotomy in the diffusive logistic model
  with a free boundary, SIAM Journal on Mathematical Analysis 42~(1) (2010)
  377--405.

\bibitem{bates2007}
P.~W. Bates, G.~Zhao, Existence, uniqueness and stability of the stationary
  solution to a nonlocal evolution equation arising in population dispersal,
  Journal of Mathematical Analysis and Applications 332~(1) (2007) 428--440.

\bibitem{wang2020free}
J.~Wang, M.~Wang, Free boundary problems with nonlocal and local diffusions
  {I}: Global solution, Journal of Mathematical Analysis and Applications
  490~(2) (2020) 123974.

\bibitem{Lifang2016}
F.~Li, J.~Coville, X.~Wang, On eigenvalue problems arising from nonlocal
  diffusion models, Discrete and Continuous Dynamical Systems 37~(2) (2017)
  879--903.

\bibitem{Coville2010}
J.~Coville, On a simple criterion for the existence of a principal
  eigenfunction of some nonlocal operators, Journal of Differential Equations
  249~(11) (2010) 2921--2953.

\bibitem{yagisita2010existence}
H.~Yagisita, Existence and nonexistence of traveling waves for a nonlocal
  monostable equation, Publications of the Research Institute for Mathematical
  Sciences 45~(4) (2010) 925--953.

\bibitem{DU2022369}
Y.~Du, W.~Ni, Spreading speed for some cooperative systems with nonlocal
  diffusion and free boundaries, part 1: Semi-wave and a threshold condition,
  Journal of Differential Equations 308 (2022) 369--420.

\end{thebibliography}

\end{document}